\documentclass[12pt]{amsart}
\usepackage{amsfonts,amssymb,amscd,xy,bm}
\usepackage[leqno]{amsmath}
\usepackage{mathrsfs}
\usepackage{amssymb}
\usepackage{amsmath}
\usepackage{amsxtra}
\usepackage{amsthm}
\usepackage{hyperref}
\numberwithin{equation}{section}

\DeclareFontEncoding{OT2}{}{} 
\newcommand{\textcyr}[1]{%
{\fontencoding{OT2}\fontfamily{wncyr}\fontseries{m}\fontshape{n}
\selectfont #1}}
\newcommand{\sha}{{\mbox{\textcyr{Sh}}}}

\usepackage[a4paper, total={5.5in, 7in}]{geometry}

\def\ofs{\mathcal O_{\lbe F\lbe,\e \si}}

\def\oks{\mathcal O_{\be K,\e \Sigma}}

\def\si{\Sigma}

\def\pic{{\rm{Pic}}\,}

\def\br{{\rm{Br}}}

\newcommand{\isoto}{\overset{\!\sim}{\to}}

\newcommand{\et}{{\rm {\acute et}}}

\DeclareMathAlphabet{\mathbbmsl}{U}{bbm}{m}{sl}

\newcommand{\fl}{\le{\rm fl}}
\newcommand{\spp}{S^{\le\lle\prime}}
\newcommand{\sppp}{S^{\le\lle\prime\prime}}

\def\sfs{S_{\fl}^{\e\sim}}

\def\sts{S_{\tau}^{\e\sim}}

\def\bg{{\mathbb G}}

\def\si{\Sigma}

\def\hom{{\rm{Hom}}\e}

\usepackage[usernames,dvipsnames]{xcolor}
\usepackage{tabularx}
\usepackage{epsfig,amssymb}
\usepackage{bm} 
\usepackage{verbatim} 

\usepackage{hyperref}
\definecolor{labelkey}{rgb}{1,0,0}

\newcommand{\spec}{\mathrm{ Spec}\,}

\newcommand{\s}{\mathscr }

\def\e{\kern 0.06em}
\def\be{\kern -.1em}
\def\le{\kern 0.03em}
\def\lle{\kern 0.015em}
\def\lbe{\kern -.025em}

\newcommand{\sh}{\kern -.4em\phantom{a}^{\mathbf{\sim}}}

\newcommand{\lra}{\longrightarrow}

\def\be{\kern -.1em}
\def\le{\kern 0.03em}
\def\lle{\kern 0.04em}
\def\lbe{\kern -.025em}

\newcommand{\Q}{{\mathbb Q}}

\newcommand{\krn}{\mathrm{Ker}\e}
\newcommand{\cok}{\mathrm{Coker}\e}

\def\e{\kern 0.08em}

\newcommand{\img}{\mathrm{Im}\e}

\newtheorem{lemma}{Lemma}[section]
\newtheorem{theorem}[lemma]{Theorem}
\newtheorem{proposition-definition}[lemma]{Proposition-Definition}
\newtheorem{corollary}[lemma]{Corollary}
\newtheorem{proposition}[lemma]{Proposition}
\theoremstyle{definition}
\newtheorem{definition}[lemma]{Definition}

\theoremstyle{remark}
\newtheorem{remark}[lemma]{Remark}

\begin{document}

\input xy     
\xyoption{all}

\title[The norm map and the capitulation kernel]{The norm map and the capitulation kernel}

\subjclass[2010]{Primary 11R29, 14F20}
	
\author{Cristian D. Gonz\'alez-Avil\'es}
\address{Departamento de Matem\'aticas, Universidad de La Serena,
La Serena, Chile} \email{cgonzalez@userena.cl}

\keywords{Capitulation kernel, relative Brauer group, Tate-Shafarevich group, N\'eron-Raynaud class group, ideal class group, quadratic Galois covers.}
	
\thanks{The author was partially supported by Fondecyt grant
1160004}

\begin{abstract} Let $f\colon \spp\to S$ be a finite and faithfully flat morphism of locally noetherian schemes of constant rank $n\geq 2$ and let $G$ be a smooth, commutative and quasi-projective $S$-group scheme with connected fibers. Under certain restrictions on $f$ and $G$, we relate the kernel of the restriction map ${\rm Res}_{\le G}^{(r+1)}\colon H^{\le r+1}\lbe(S_{\et},G\e)\to H^{\le r+1}\lbe(\spp_{\et},G\e)$ in \'etale cohomology, where $r\geq 0$, to a quotient of the kernel of the mod $n$ corestriction map ${\rm Cores}_{\le G}^{(r)}\!/n\colon H^{\le r}\lbe(\spp_{\et},G\e)/n\to H^{\le r}\lbe(S_{\et},G\e)/n$. When $r=0$ and $f$ is a Galois covering with Galois group $\Delta$, our main theorem relates $\krn\e {\rm Res}_{\le G}^{(1)}=H^{\le 1}\be(\Delta,G(\spp\le)\e)$ to the subgroup of $G(\spp\le)$ of sections whose $(\spp\!/\lbe S\e)$-norm lies in $G(S\e)^{n}$. Applications are given to the capitulation problem for N\'eron-Raynaud class groups of tori and Tate-Shafarevich groups of abelian varieties. 
\end{abstract}
	
\maketitle

\section{Introduction}

Let $f\colon \spp\to S$ be a finite and faithfully flat morphism of locally noetherian schemes of constant rank $n\geq 2$ and let $G$ be a smooth, commutative and quasi-projective $S$-group scheme with connected fibers. In this paper we relate the kernel of the restriction map ${\rm Res}_{\le G}^{(r+1)}\colon H^{\le r+1}\lbe(S_{\et},G\e)\to H^{\le r+1}\lbe(\spp_{\et},G\e)$, where $r\geq 0$, to a certain quotient of the kernel of the mod $n$ corestriction (or norm) map ${\rm Cores}_{\le G}^{(r)}\!/n\colon H^{\le r}\lbe(\spp_{\et},G\e)/n\to H^{\le r}\lbe(S_{\et},G\e)/n$ when $(f,G\e)$ is {\it admissible}, which means that the following conditions hold at every point $s\in S$ such that ${\rm char}\, k(\lbe s\lbe)$ divides $n\e$: (i) $f_{\lbe s}\colon \spp\times_{S}\spec k(\lbe s\lbe)\to \spec k(\lbe s\lbe)$ is \'etale (i.e., unramified), and (ii) $G_{\be\lle s}$ is a semiabelian $k(\lbe s\lbe)$-variety. For example, $(f,G\e)$ is admissible if $f$ is \'etale and $G$ is a semiabelian $S$-scheme. See below for more examples. In order to simplify the exposition in this Introduction, we state below a corollary of our main theorem (Theorem \ref{main}) in the following case: $r=0$, $f$ is a Galois covering with Galois group $\Delta$ (in particular, $f$ is \'etale) and $G=\bg_{m,\e S}$ is the multiplicative group scheme over $S$. In this case $\krn\e{\rm Res}_{\le G}^{(1)}\simeq\krn\e\pic f\simeq H^{\le 1}\be(\Delta,U\be(\lbe\spp\e)\e)$, where $\pic\be f\colon\pic S\to \pic \spp$ is the canonical map induced by $f$ and $U\be(\lbe\spp\e)=\varGamma(\spp\be,\lbe \mathcal O_{\lbe S^{\prime}}\be)^{*}$ is the group of global units on $\spp\,$\footnote{\e The isomorphism $\krn\e\pic f\simeq H^{\le 1}\be(\Delta,U\be(\lbe\spp\e)\e)$ follows from the Hochschild-Serre spectral sequence associated to the Galois covering $\spp\be/\lbe S\e$.}. Consider the following abelian fppf sheaf on $S$:
\begin{equation}\label{gen0}
\bg_{m}\lbe(n)=R_{\e S^{\lle\prime}\be/ S}^{\e(1)}(\e\mu_{\e n})/\mu_{\e n,\e S},
\end{equation}
where $R_{\e S^{\lle\prime}\be/ S}^{\e(1)}(\e\mu_{\e n})=\krn[\e N_{\spp\be/S}\colon\be R_{\e S^{\lle\prime}\be/ S}(\e\mu_{\e n,\e S^{\lle\prime}}\be)\to \mu_{\e n,\e S}]$ is the $S$-group scheme of $n$-th roots of unity on $\spp$ of $(\spp\!/\lbe S\e)$-norm $1$. The sheaf \eqref{gen0} is an $(\le \spp\be/S\e)$-form of $\mu_{\e n,\e S}^{\lle n-2}$, i.e., $\bg_{m}(n)_{S^{\lle\prime}}\simeq \mu_{\lle n,\e S^{\lle\prime}}^{\lle n-2}$. Now set 
\begin{equation}\label{fnu}
\Psi_{\be N}(n,U\e)=\{x\in U\be(\lbe\spp\le)\colon N_{S^{\lle\prime}\!\lbe/S}(x)\in U\be(\lbe S\e)^{n}\}.
\end{equation}
Then the following holds

\begin{theorem} \label{uno} There exists a canonical exact sequence of $n$-torsion abelian groups
\[
\begin{array}{rcl}
1&\to& U\be(\lbe S\e)\be\cap\be U\be(\lbe\spp\le)^{n}\be/\le U\be(\lbe S\e)^{n}\to H^{\le 1}\lbe(\Delta, \mu_{\lle n}\lbe(\lbe\spp\lle)\lbe)\to \krn\e \pic\be f\\
&\to& \Psi_{\be N}(n,U\e)/\le U\be(\lbe S\e)\e U\be(\lbe\spp\le)^{n}\to H^{\le 1}\lbe (\varGamma^{\e\bullet}_{\!n}\lbe(U\e)),
\end{array}
\]
where $\Psi_{\be N}(n,U\e)$ is the group \eqref{fnu} and
\[
H^{\le 1}\lbe (\varGamma^{\e\bullet}_{\!n}\lbe(U\e))=\frac{\krn[\e H^{\le 1}\be(S_{\fl},\bg_{m}\be(n))\to H^{\le 2}\lbe(S_{\fl}, \mu_{\le n}\lbe)\e]}{\img[\e H^{\le 0}\lbe(S_{\fl}, \mu_{\le n}\lbe)\to H^{\le 1}\lbe(S_{\fl},\bg_{m}\be(n))\e]},
\]
where $\bg_{m}\lbe(n)$ is the sheaf \eqref{gen0}.
\end{theorem}

If $\Delta$ is cyclic, then the above theorem is neither surprising nor optimal since the group
\[
\krn\e \pic\be f\simeq H^{\le 1}\be(\Delta,U\be(\lbe\spp\le)\e)\simeq H^{\le -1}\be(\Delta,U\lbe(\lbe\spp\le)\e) \quad\text{(Tate cohomology)}
\]
is a quotient of $\{x\in U\lbe(\lbe\spp\le)\colon N_{S^{\lle\prime}\!\lbe/S}(x)=1 \}$, whence $\krn\e \pic\be f$ can be directly related to a subgroup of $U\lbe(\lbe\spp\le)$ defined in terms of the norm map $N_{S^{\lle\prime}\!/S}$. Note, however, that the theorem shows (in particular) that $\krn\e \pic\be f$ can be related to such a subgroup of $U\lbe(\lbe\spp\le)$ for {\it any} $\Delta$.

Another class of admissible pairs is the following. Let $F$ be a number field, write  $\mathcal O_{\lbe F}$ for the ring of integers of $F$ and let $K\lbe/\lbe F$ be a finite Galois extension of degree $n$ which is unramified over the set $\Omega$ of primes of $F$ that divide $n$. Let $G_{\be F}$ be either an abelian variety or an algebraic torus over $F$ with semiabelian reduction over $\Omega\,$\e\footnote{\e  Thus, if $G_{\be F}$ is an $F$-torus, then $G_{\be F}$ has multiplicative (i.e., toric) reduction over $\Omega$.}. If $f\colon\spec \mathcal O_{\lbe K}\to \spec\mathcal O_{\lbe F}$ is the morphism induced by the inclusion $\mathcal O_{\lbe F}\subset \mathcal O_{\lbe K}$ and $G$ denotes the identity component of the N\'eron-Raynaud model of $G_{\be F}$ over $\mathcal O_{\lbe F}$, then $(\le f,G\e)$ is an admissible pair. In this setting our main Theorem \ref{main} yields information on the {\it Capitulation Problem} for the Tate-Shafarevich group of $G_{\be F}$ over $F$ and the N\'eron-Raynaud class group of $G_{\be F}$ introduced in \cite[\S3]{ga12}. See Section \ref{new} for the details.

\smallskip

The paper is organized as follows. Section \ref{pre} consists of preliminaries, in particular on the Weil restriction functor. In section \ref{nm1} we discuss the norm one group scheme, which plays a central role in the paper. In section \ref{mid} we establish our main theorem (Theorem \ref{main}). The developments of this section were inspired by \cite[\S3]{ps} but, in contrast to [loc.cit.], we avoid working with hypercohomological spectral sequences and work instead primarily with complexes of length 3. In section \ref{quad} we specialize our results to the (interesting) case of quadratic Galois coverings. Section \ref{new}, which concludes the paper, discusses the applications of our main theorem to the Capitulation Problem mentioned above for quadratic Galois extensions of global fields.

\section*{Acknowledgements}
I thank Laurent Moret-Bailly for answering my question \cite{mb}. I also thank the referee for pointing out the need to include additional applications of the main theorem, which  motivated the inclusion of Section \ref{new} in this version. This research was partially supported by Fondecyt grant 1160004.

\section{Preliminaries}\label{pre}

If  $n\geq 1$ is an integer and $A$ is an object of an abelian category $\s A$, $A_{\e n}$ (respectively, $A/ n\le$) will denote the kernel (respectively, cokernel) of the multiplication by $n$ morphism on $A$. If $\psi\colon A\to B$ is a morphism in $\s A$, $\psi_{\le n}\colon A_{\e n}\to B_{\le n}$ and $\psi/n\colon A/n\to B/n$ will denote the morphisms in $\s A$ induced by $\psi$. 
Note that 
\begin{equation} \label{co1}
\krn (\psi_{\le n})=(\krn\psi)_{n}\qquad\text{and}\qquad \cok (\psi/n)=(\cok\psi)/n.
\end{equation}

\begin{proposition}\label{ker-cok}  Let  $A\overset{\!f}{\to}B\overset{\!g}{\to}C$ be  morphisms in an abelian category $\s A$. Then there exists a canonical exact sequence in $\s A$
\[
0\to \krn f\to \krn\lbe(\e g\be\circ\be f\e)\to \krn g\to\cok f\to \cok\lbe(\e g\be\circ\be f\e) \to \cok g\to 0.
\]
\end{proposition}
\begin{proof} See, for example, \cite[1.2]{bey}. The middle map $\krn g\to\cok f$ is the composition $\krn g\hookrightarrow B\twoheadrightarrow \cok f$. The remaining maps are the natural ones.
\end{proof}

{\it All schemes below are tacitly assumed to be non-empty.}

\smallskip

If $S$ is a scheme and $\tau\, (=\et$ or $\fl$) denotes either the \'etale or the fppf topology on $S$, $S_{\tau}$ will denote the small $\tau$ site over $S$. A faithfully flat morphism locally of finite presentation $\spp\to S$ is an fppf covering of $S$. We will write $\sts$ for the category of sheaves of abelian groups on $S_{\tau}$, which is abelian. If $G$ is a commutative $S$-group scheme, the presheaf represented by $G$ is an object of $S_{\tau}^{\e\sim}$. In particular, if $f\colon\spp\to S$ is an fppf covering of $S$ as above, then the map $G(S\e)\hookrightarrow G(\spp)$ induced by $f$ is an injection that will be regarded as an inclusion. If $n\geq 1$ is an integer, the object $G_{n}$ of the abelian category $\sts$ is represented by the $S$-group scheme $G\!\times_{n_{\le G},\e G,\e\varepsilon}\! S$, where $n_{\le G}$ is the $n$-th power morphism on $G$ and $\varepsilon\colon S\to G$ is the unit section of $G$. We will make the identifications 
\[
n_{\le G_{\be S^{\lle\prime}}}=n_{\le G}\times_{S}\spp \qquad\text{and}\qquad G_{\lbe n,\e S^{\lle\prime}}=(G_{\lbe n})_{S^{\lle\prime}}=(G_{\lbe S^{\lle\prime}}\lbe)_{n}.
\]
If $G$ is separated over $S$, then $G_{n}\hookrightarrow G$ is a closed immersion. If, in addition, $G$ is quasi-projective over $S$, then $G_{n}$ is also quasi-projective over $S$ \cite[II, Proposition 5.3.4(i)]{ega}. If $\psi\colon G\to H$ is a morphism of commutative $S$-group schemes, $\psi_{n}\colon G_{n}\to H_{n}$ will denote the morphism of $S$-group schemes induced by $\psi$ \cite[(1.2.3), pp.~26-27]{ega1}.

If $G$ is as above, we will write $H^{\le r}\be(\spp_{\tau},G\e)$ for $H^{\le r}\be(\spp_{\tau},G_{\lbe S^{\lle\prime}}\be)$, where $G_{\lbe S^{\lle\prime}}=G\times_{S}\spp$. If $G$ is smooth over $S$, $H^{\le r}\lbe(S_{\fl},G\e)$ and $H^{\le r}\lbe(S_{\et},G\e)$ will be identified via \cite[Theorem 11.7(1), p.~180]{dix}. 
If $G=\bg_{m,\e S}$, the groups $H^{\le r}\lbe(S_{\tau},G\e)$ will be denoted by $H^{\le r}\lbe(S_{\tau},\bg_{m}\lbe)$. Further, we will identify $H^{\le 1}(S_{\et},\bg_{m})$ and $\pic S$ via  \cite[Theorem 4.9, p.~124]{mi1} and we will write $\br^{\e\prime}S$ for the cohomological Brauer group of $S$, i.e., $\br^{\e\prime}S=H^{\le 2}(S_{\et},\bg_{m})$. Let $\br\e S$ denote the Brauer group of equivalence classes of Azumaya algebras on $S$. If $S$ is quasi-compact and admits an ample invertible sheaf, then the canonical map 
$\br\e S\to \br^{\e\prime}\e S$ induces an isomorphism of torsion abelian groups
\begin{equation}\label{br}
\br\e S\simeq (\br^{\e\prime}S\e)_{\rm tors}
\end{equation}
by \cite{j} and \cite[Tag 01PR, Lemma 27.26.8]{sp}.

A sequence of commutative $S$-group schemes
\begin{equation}\label{sex}
0\to G_{1}\to G_{2}\to G_{3}\to 0
\end{equation}
will be called {\it exact} if the corresponding sequence of representable objects in $\sfs$ is exact. If $G_{2}\to G_{3}$ is faithfully flat and locally of finite presentation and $G_{\lbe 1}=G_{2}\times_{G_{3}}S$ denotes the scheme-theoretic kernel of $G_{2}\to G_{3}$, then \eqref{sex} is exact.

\smallskip

Let $f\colon S^{\e\prime}\to S$ be a finite and faithfully flat morphism of locally noetherian schemes of constant rank $n\geq 2$ and let $X^{\prime}$ be an $\spp$-scheme. The {\it Weil restriction of $X^{\prime}$ along $f$} is the contravariant functor
$(\mathrm{Sch}/S)\to(\mathrm{Sets})$ defined by $T\mapsto\hom_{
S^{\le\prime}}(T\times_{S}S^{\le\prime},X^{\le\prime}\le)$. This functor is representable if there exist an $S$-scheme $R_{S^{\le\prime}\be/S}(X^{\prime}\e)$ and a morphism of $S^{\e\prime}$-schemes
\begin{equation}\label{the}
\theta_{X^{\prime},\e S^{\le\prime}\be/S}\colon R_{S^{\le\prime}\be/S}(X^{\prime}\e)_{S^{\le\prime}}\to X^{\prime}
\end{equation}
such that the map
\begin{equation}\label{wr}
\hom_{\le S}\e(T,R_{S^{\le\prime}\be/S}(X^{\le\prime}\e))\to\hom_{
S^{\le\prime}}(\e T\!\times_{S}\!S^{\e\prime},X^{\le\prime}\e), \quad g\mapsto \theta_{ X^{\prime}\!,\, S^{\le\prime}\be/S}\circ g_{\le S^{\le\prime}},
\end{equation}
is a bijection (functorially in $T\e$). See \cite[\S7.6]{blr} and \cite[Appendix A.5]{cgp} for basic information on the Weil restriction functor\,\footnote{\e As noted by Brian Conrad, the restriction in \cite[Appendix A.5]{cgp} to an affine base $S$ can be removed since all assertions in [loc.cit.] are local on $S$. Further, the noetherian hypotheses in [loc.cit.] are satisfied since we are assuming that $S$ and $\spp$ are locally noetherian. Finally, we note that the quasi-projectivity hypotheses in [loc.cit.] are only needed to guarantee the existence of the relevant Weil restrictions in the category of schemes.}\,. The map $\theta_{X^{\prime}\be,\e S^{\le\prime}\be/S}$ \eqref{the} is both functorial in $X^{\prime}$ and compatible with compositions and arbitrary base changes \cite[(2.40) and (2.41), p.~16]{gfr}. We will write
\begin{equation}\label{imor}
j_{ X,\e S^{\lle\prime}\be/ S}\colon X\to R_{\e S^{\lle\prime}\be/ S}(X_{\be S^{\lle\prime}}\be)
\end{equation}
for the canonical adjunction $S$-morphism, i.e., the $S$-morphism that corresponds to the identity morphism of $X_{S^{\lle\prime}}$ under the bijection \eqref{wr}. 

\smallskip

For lack of an adequate reference, we include here a proof of

\begin{proposition}\label{wtf} Let $n\geq 1$ be an integer and let $S_{i}^{\e\prime}\to S$ (where $1\leq i\leq n$) be a finite collection of $S$-schemes. Set $S^{\e\prime}=\coprod_{\e i=1}^{\e n}S_{ i}^{\e\prime}$ and let $\spp\to S$ be the canonical morphism of schemes induced by the given morphisms $S_{i}^{\e\prime}\to S$. Then, for every $S^{\e\prime}$-scheme $X^{\lle\prime}$, there exists an isomorphism of functors
\[
R_{\e S^{\e\prime}\!/ S}(X^{\le\prime}\le)\isoto\prod_{\e i=1}^{\e n}R_{\e S_{i}^{\lle\prime}\lbe/S}(X^{\le\prime}\!\times_{S^{\lle\prime}}\be S_{i}^{\e\prime}\e).
\]
\end{proposition}
\begin{proof} Let $Y$ be any $S$-scheme. Since the tensor product is distributive with respect to finite direct products, \eqref{wr} and \cite[\S3.1, pp.~230-231]{ega1} yield canonical bijections
\[
\begin{array}{rcl}
\hom_{S}(\e Y,R_{\e S^{\e\prime}\!/\e S}(X^{\e\prime}\le))&\isoto& \hom_{S^{\le\prime}}(\e Y\!\times_{S}S^{\e\prime},X^{\lle\prime}\le)\\
&\simeq&\hom_{S^{\le\prime}}(\e\coprod_{\e i=1}^{\e n}(\e Y\!\times_{S}S_{i}^{\e\prime}\e),X^{\e\prime}\le)\\
&\simeq&\prod_{\e i=1}^{\e n}\hom_{S^{\le\prime}}(\e Y\!\times_{S}S_{i}^{\e\prime},X^{\e\prime}\le)\\
&\simeq&\prod_{\e i=1}^{\e n}\hom_{S_{i}^{\le\prime}}(\e Y\!\times_{S}S_{i}^{\e\prime},X^{\e\prime}\times_{S}S_{i}^{\e\prime}\le)\\
&\simeq&\prod_{\e i=1}^{\e n}\hom_{S}(\e Y,R_{\e S_{i}^{\e\prime}\lbe/\le S}(X^{\e\prime}\!\times_{S^{\e\prime}}\be S_{i}^{\e\prime}\e))\\
&\simeq&\hom_{S}(\e Y,\prod_{\e i=1}^{\e n}\!R_{\e S_{i}^{\e\prime}\lbe/\le S}(X^{\e\prime}\!\times_{S^{\e\prime}}\be S_{i}^{\e\prime}\e)).
\end{array}
\]
The proposition follows.
\end{proof}

\smallskip

If $n\geq 1$ is an integer and $X$ is an $S$-scheme, $X^{n}$ will denote the $S$-scheme defined recursively by $X^{1}=X$ and $X^{n}=X\times_{\be S} X^{n-1}$ for $n\geq 2$.

\begin{corollary}\label{mb} Let $X$ be an $S$-scheme such that the $S$-scheme $R_{\e S^{\e\prime}\!/\lbe S}(X_{\lbe S^{\lle\prime}}\be)$ exists and let $S^{\e\prime\prime}\to S$ be a morphism of schemes such that $S^{\e\prime}\times_{S}S^{\e\prime\prime}\simeq \coprod_{\e i=1}^{\e n}\! S^{\e\prime\prime}$, where $n\geq 1$. Then there exists a canonical isomorphism of $S^{\e\prime\prime}$-schemes
\[
R_{\e S^{\lle\prime}\!/ S}(X_{S^{\lle\prime}}\lbe)\times_{S}S^{\e\prime\prime}\simeq X^{ n}_{\lbe S^{\lle\prime\prime}}\e.
\]	
\end{corollary}
\begin{proof} This is immediate from the proposition via the canonical isomorphism of $S^{\e\prime\prime}$-schemes $R_{\e S^{\lle\prime}\be/\e S}(X_{S^{\lle\prime}}\e)\be\times_{S}\be S^{\lle\prime\prime}\simeq R_{\e S^{\lle\prime}\lbe\times_{\lbe S}\lbe S^{\lle\prime\prime}\!/\e S^{\lle\prime\prime}}(X_{S^{\lle\prime}\lbe\times_{\lbe S}  S^{\lle\prime\prime}}\lbe)$.
\end{proof}

\smallskip

\section{The norm one group scheme}\label{nm1}

Let $f\colon \spp\to S$ be a finite and faithfully flat morphism of locally noetherian schemes of constant rank $n\geq 2$ and let $G$ be a commutative and quasi-projective $S$-group scheme\,\footnote{\, Note that, if $S$ is the spectrum of a field, then every $S$-group scheme of finite type is quasi-projective over $S$ by \cite[Proposition A.3.5, p.~486]{cgp}.} with unit section $\varepsilon\colon S\to G$. Then $R_{\e S^{\lle\prime}\be/ S}(G_{\lbe S^{\lle\prime}}\be)$ is a commutative and quasi-projective $S$-group scheme with unit section $R_{\e S^{\lle\prime}\be/ S}(\varepsilon_{S^{\lle\prime}}\be)$. See \cite[\S7.6, Theorem 4, p.~194]{blr}, \cite[II, Corollary 4.5.4]{ega} and \cite[Proposition A.5.8, p.~513]{cgp}. Further, by \cite[Proposition A.5.2(3), p.~506]{cgp}, the $S$-morphisms $R_{\e S^{\lle\prime}\be/ S}(n_{\lle G_{\be\lle S^{\prime}}}\be)$ and $n_{R_{S^{\lle\prime}\!/\lbe S}(G_{\be\lle S^{\prime}}\lbe)}$ can be identified, whence $R_{\e S^{\lle\prime}\be/ S}(G_{\be S^{\prime}}\lbe)_{n}=R_{\e S^{\lle\prime}\be/ S}(G_{\lbe n,\e S^{\lle\prime}}\be)$.
Now $j_{\e G,\e S^{\lle\prime}\be/ S}\colon G\hookrightarrow R_{\e S^{\lle\prime}\be/ S}(G_{\lbe S^{\lle\prime}}\lbe)$ \eqref{imor} is a closed immersion \cite[Proposition A.5.7, p.~510]{cgp} which induces a closed immersion of commutative and quasi-projective $S$-group schemes $(\e j_{\e G,\e S^{\lle\prime}\be/ S})_{n}=j_{\e G_{n},\e S^{\lle\prime}\be/ S}\colon G_{\lbe n}\hookrightarrow  R_{\e S^{\lle\prime}\be/ S}(G_{\lbe n,\e S^{\lle\prime}}\be)$. We also note that the composition of $\spp$-morphisms
\begin{equation}\label{ncom}
G_{\lbe S^{\le\prime}}\overset{\! (\e\le j_{G\lbe,\lle S^{\lle\prime}\!/\lbe S})_{\lbe S^{\le\prime}}}{\hookrightarrow } R_{\e S^{\lle\prime}\be/ S}(G_{\lbe S^{\lle\prime}}\lbe)_{S^{\le\prime}}\overset{\!\theta_{G_{\be S^{\prime}}\lbe,\e S^{\le\prime}\!/\lbe S}}{\lra}G_{\lbe S^{\le\prime}}
\end{equation}
equals $n_{\lle G_{\be\lle S^{\prime}}}=(n_{\lle G})_{S^{\prime}}$, where $\theta_{G_{\be S^{\prime}}\lbe,\e S^{\le\prime}\!/\lbe S}$ is given by \eqref{the} (see, e.g., \cite[p.~511, lines 23-31]{cgp}).

\begin{lemma}\label{bic} If $G$ is a smooth, commutative and quasi-projective $S$-group scheme with connected fibers, then $R_{\e S^{\lle\prime}\be/ S}(G_{\lbe S^{\lle\prime}}\lbe)$ is (respectively) a smooth, commutative and quasi-projective $S$-group scheme with connected fibers. 
\end{lemma}
\begin{proof} The existence, commutativity and quasi-projectivity of $R_{\e S^{\lle\prime}\be/ S}(G_{\lbe S^{\lle\prime}}\lbe)$ has been noted above. For the smoothness and connectedness of its fibers, see \cite[\S7.6, Proposition 5, p.~195]{blr}, \cite[Proposition A.5.9, pp.~514]{cgp} and \cite[${\rm VI_{A}}$, Proposition 2.1.1]{sga3}. 
\end{proof}

Next let 
\begin{equation}\label{tmor}
N_{G,\e S^{\lle\prime}\be/ S}\colon R_{\e S^{\lle\prime}\be/ S}(G_{\lbe S^{\lle\prime}}\be)\to G
\end{equation}
be the norm morphism  defined in \cite[XVII, 6.3.13.1 and 6.3.14(a)]{sga4}. By \cite[XVII, Proposition 6.3.17]{sga4}, \eqref{tmor} is uniquely determined by the properties of being functorial in $G$, compatible with compositions and arbitrary base changes and by the fact that the composition
\begin{equation}\label{nis}
n_{\le G}\colon G\overset{j_{\e G,\le S^{\lle\prime}\be/ S}}{\hookrightarrow } R_{\e S^{\lle\prime}\be/ S}(G_{\lbe S^{\lle\prime}}\be)\overset{\!N_{G,\e S^{\lle\prime}\be/S}}{\lra} G
\end{equation}
is the $n$-th power morphism on $G$.

\begin{proposition}\label{nsm} If $G$ is smooth over $S$, then $N_{G,\e S^{\lle\prime}\be/S}$ \eqref{tmor} is smooth and surjective.
\end{proposition}
\begin{proof} By \cite[Proposition A.5.11(1), p.~516]{cgp}, $\theta_{\le G_{\be S^{\prime}},\e S^{\le\prime}\be/S}\colon R_{\e S^{\lle\prime}\be/ S}(G_{\lbe S^{\lle\prime}}\be)_{S^{\lle\prime}}\to G_{S^{\lle\prime}}$ \eqref{the} is a smooth and surjective morphism of $\spp$-group schemes. Now, if we identify
$R_{\e S^{\lle\prime}\be/ S}(G_{\lbe S^{\lle\prime}}\lbe)_{S^{\le\prime}}$ and
$R_{\e S^{\lle\prime\prime}\be/ S^{\lle\prime}}(G_{\lbe S^{\lle\prime\prime}}\lbe)$ via \cite[(2.40), p.~16]{gfr}, where $S^{\lle\prime\prime}=S^{\lle\prime}\times_{S}S^{\lle\prime}$, then $(\e j_{\le G,\lle S^{\lle\prime}\!\lbe/\lbe S})_{S^{\le\prime}}\colon
G_{\lbe S^{\le\prime}}\to R_{\e S^{\lle\prime}\be/ S}(G_{\lbe S^{\lle\prime}}\lbe)_{S^{\le\prime}}$ is identified with
$j_{\e G_{\lbe S^{\lle\prime}},\le S^{\lle\prime\prime}\!\lbe/\lbe S^{\lle\prime}}$ and $\theta_{\le G_{\lbe S^{\prime}},\e S^{\le\prime}\be/S}$ is identified with a map $R_{\e S^{\lle\prime\prime}\!/ S^{\lle\prime}}(G_{\lbe S^{\lle\prime\prime}}\lbe)\to G_{\lbe S^{\lle\prime\prime}}$ that has the same properties that characterize $N_{\e G_{\lbe S^{\lle\prime}},\e S^{\lle\prime\prime}\!\lbe/\lbe S^{\lle\prime}}$. See \eqref{ncom} and \eqref{nis}. Thus we may identify $\theta_{\le G_{\be S^{\prime}},\e S^{\le\prime}\be/S}$ and $N_{\e G_{\lbe S^{\lle\prime}},\e S^{\lle\prime\prime}\!\lbe/\lbe S^{\lle\prime}}$, whence $(N_{G, \e S^{\lle\prime}\be/ S})_{\lbe S^{\lle\prime}}=N_{\e G_{\lbe S^{\lle\prime}},\e S^{\lle\prime\prime}\!\lbe/\lbe S^{\lle\prime}}$ is a smooth and surjective morphism of $\spp$-group schemes. The proposition now follows from  \cite[${\rm IV}_{4}$, Corollary 17.7.3(ii)]{ega} and \cite[Proposition 3.6.4, p.~245]{ega1}.
\end{proof}

The {\it norm one group scheme associated to $(f, G\e)$} is the $S$-group scheme
\[
R_{\e S^{\lle\prime}\!/\lbe S}^{\e(1)}\lbe(G\e)=\krn[R_{\e S^{\lle\prime}\be/ S}(G_{\lbe S^{\lle\prime}}\be)\overset{\!\be N_{\lbe G,\le S^{\lle\prime}\!\lbe/ S}}{\lra} G\,],
\]
where $N_{\lbe G,\le S^{\lle\prime}\!\lbe/ S}$ is the norm morphism \eqref{tmor}. 
If $G$ is smooth over $S$, then $R_{\e S^{\lle\prime}\!/\lbe S}^{\e(1)}(G\e)$ is smooth over $S$ by Proposition \ref{nsm}. Further, in this case $N_{\lbe G,\le S^{\lle\prime}\!\lbe/ S}$ is faithfully flat and the sequence of $S$-group schemes
\[
0\to R_{\e S^{\lle\prime}\!/\lbe S}^{\e(1)}\lbe(G\e)\to R_{\e S^{\lle\prime}\be/ S}(G_{\lbe S^{\lle\prime}}\be)\overset{\!\be N_{\lbe G,\le S^{\prime}\!\lbe/\lbe S}}{\lra} G\to 0
\]
is exact relative to the \'etale topology on $S$.

\begin{lemma} \label{vnice} If $S^{\e\prime}\simeq \coprod_{\e i=1}^{\e n}S$, then there exist canonical isomorphisms of $S$-group schemes $R_{\e S^{\lle\prime}\!/\lbe S}\lbe(G\e)\isoto G^{\e n}$ and $R_{\e S^{\lle\prime}\!/\lbe S}^{\e(1)}\lbe(G\e)\isoto G^{\e n-1}$.
\end{lemma}
\begin{proof} The first isomorphism follows from Corollary \ref{mb} (with $\sppp=S$ there). Under this isomorphism, the norm morphism $N_{\lbe G,\le S^{\lle\prime}\!\lbe/ S}\colon R_{\e S^{\lle\prime}\!/\lbe S}\lbe(G\e)\to G$ corresponds to the product morphism $G^{\e n}\to G$. See  \cite[XVII, Proposition 6.3.15(iii)]{sga4}. The second isomorphism then follows.
\end{proof}

\begin{proposition} \label{wrp2} Let $\spp\to S$ correspond to $B/F$, where $F$ is a field and $B$ is a nonzero, finite and \'etale $F$-algebra of rank $n\geq 2$. Then $R_{\e B/F}(G\e)$ and $R_{\e B/F}^{\e(1)}(G\e)$ are forms of $G^{\le\lle n}$ and $G^{\le\lle n-1}$, respectively.
\end{proposition}
\begin{proof} By \cite[$\text{IV}_{4}$, Corollary 17.4.2$({\text d}^{\le\prime}\e)$]{ega}, there exists an isomorphism of $F$-algebras $B\simeq \prod_{\e i=1}^{\e r} K_{\le i}$ for some positive integer $r$, where each $K_{i}$ is a finite and separable extension of $F$. Set $f_{i}=[K_{i}\le\colon\! F\,]$, so that $n={\rm rank}_{F}B=\sum_{\e i=1}^{\e r}f_{i}$. We now choose a separable closure $F^{\e\rm s}$ of $F$ containing $K_{i}$ and $F^{\e\rm s}$-isomorphisms $K_{i}\!\otimes_{F}\be F^{\e\rm s}\simeq (F^{\e\rm s})^{f_{i}}$ for $i=1,\dots,r$. Then	$B\otimes_{F}\be F^{\e\rm s}\simeq \prod_{\e i=1}^{\e r} (K_{i}\!\otimes_{F}\be F^{\e\rm s})\simeq \prod_{\e i=1}^{\e r}(F^{\e\rm s})^{f_{i}}\simeq (F^{\e\rm s})^{n}$, whence $\spec B\times_{F}\spec F^{\e\rm s}\simeq \coprod_{\e i=1}^{\e n}\spec F^{\e\rm s}$. Now Lemma \ref{vnice} yields isomorphisms of $F^{\e\rm s}$-\e group schemes $R_{\e B/F}\lbe(G\e)\be\times_{\lbe F}\lbe \spec F^{\e\rm s}\simeq G^{\le\lle n}_{\be F^{\lle\rm s}}$ and
$R_{\e B/F}^{\e(1)}\lbe(G\e)\be\times_{\lbe F}\lbe \spec F^{\e\rm s}\simeq G^{\le\lle n-1}_{\be F^{\lle\rm s}}$, as claimed.	
\end{proof}

\begin{lemma} \label{exa} Let $\spp\to S$ correspond to $B/F$, where $F$ is a field and $B$ is a finite $F$-algebra. Let $0\to G^{\e\prime}\to G\to G^{\e\prime\prime}\to 0$ be an exact sequence of smooth, commutative and quasi-projective $F$-group schemes. Then the given sequence induces exact sequences of smooth, commutative and quasi-projective $F$-group schemes
\[
0\to R_{\e B\lbe/\lbe F}\lbe(G_{\be B}^{\e\prime}\e)\to R_{\e B\lbe/\lbe F}\lbe(G_{\be B})\to R_{\e B\lbe/\lbe F}\lbe(G_{\be B}^{\e\prime\prime}\e)\to 0
\]
and
\[
0\to R_{\e B\lbe/\lbe F}^{\e(1)}\lbe(G^{\e\prime}\e)\to R_{\e B\lbe/\lbe F}^{\e(1)}\lbe(G\e)\to R_{\e B\lbe/\lbe F}^{\e(1)}\lbe(G^{\e\prime\prime}\e)\to 0.
\]
\end{lemma}
\begin{proof} The first sequence is exact by \cite[Proposition A.5.4(3), p.~508]{cgp}. The second sequence follows by applying the snake lemma to the following exact and commutative diagram in $(\spec F\e)_{\fl}^{\le\sim}\e$, whose vertical arrows are surjections by Proposition \ref{nsm}\,:
\[
\xymatrix{
0\ar[r]& R_{\e B\lbe/\lbe F}\lbe(G_{\be B}^{\e\prime}\e)\ar@{->>}[d]^(.45){N_{G^{\lle\prime}\!,\le B\lbe/\lbe F}}
\ar[r]& R_{\e B\lbe/\lbe F}\lbe(G_{\be B})\ar[r]\ar@{->>}[d]^(.45){N_{G\!,\le B\lbe/\lbe F}}& R_{\e B\lbe/\lbe F}\lbe(G_{\be B}^{\e\prime\prime}\e)\ar@{->>}[d]^(.45){N_{G^{\lle\prime\prime}\!,\le B\lbe/\lbe F}}\ar[r]& 0\\
0\ar[r]& G^{\e\prime}\ar[r]& G\ar[r]& G^{\e\prime\prime}\ar[r]& 0.
}
\]
\end{proof}

\begin{proposition} \label{radm} Let $\spp\to S$ correspond to $B/F$, where $F$ is a field and $B$ is a nonzero, finite and \'etale $F$-algebra. If $G$ is a torus (respectively, abelian variety, semiabelian variety) over $F$, then $R_{\e B/F}(G\e)$ and $R_{\e B/F}^{\e(1)}(G\e)$ are tori (respectively, abelian varieties, semiabelian varieties) over $F$.
\end{proposition}
\begin{proof} If $G$ is a torus (respectively, abelian variety) over $F$, then so also are $R_{\e B/F}(G\e)$ and $R_{\e B/F}^{\e(1)}(G\e)$ by Proposition \eqref{wrp2}. Now, if $G$ is a semiabelian variety over $F$ given as an extension $1\to T\to G\to A\to 1$, where $T$ (respectively, $A$) is a torus (respectively, abelian variety) over $F$, then Lemma \ref{exa} shows that $R_{\e B/F}(G\e)$ and $R_{\e B/F}^{\e(1)}(G\e)$ are semiabelian varieties over $F$ given as extensions
\[
0\to R_{\e B\lbe/\lbe F}\lbe(\e T\e)\to R_{\e B\lbe/\lbe F}\lbe(G\e)\to R_{\e B\lbe/\lbe F}\lbe(A\e)\to 0
\] 
and
\[
0\to R_{\e B\lbe/\lbe F}^{\e(1)}\lbe(\e T\e)\to R_{\e B\lbe/\lbe F}^{\e(1)}\lbe(G\e)\to R_{\e B\lbe/\lbe F}^{\e(1)}\lbe(A\e)\to 0.
\]
\end{proof}

\begin{lemma} \label{rad} Let $\spp\to S$ correspond to $K\be/\be F$, where $F$ is a field and $K$ is a finite and purely inseparable extension of $F$. If $G$ is a smooth, commutative and quasi-projective $F$-group scheme, then $R_{\e K/F}^{\e(1)}\lbe(G\e)$ is a smooth, connected and unipotent $F$-group scheme.
\end{lemma}
\begin{proof} The smoothness of $R_{\e K/F}^{\e(1)}\lbe(G\e)$ has already been noted. By the proof of Proposition \ref{nsm}, the maps $(N_{G,\le K/F})_{K}$ and $\theta_{G_{\lbe K},\e K/F}$ can be identified. Now, by \cite[Proposition A.5.11(2), p.~517]{cgp}, $\krn\e\theta_{G_{\lbe K},\e K/F}$ is connected and unipotent. Thus $R_{\e K/F}^{\e(1)}\lbe(G\e)_{K}$ is connected and unipotent over $K$, whence $R_{\e K/F}^{\e(1)}\lbe(G\e)$ is connected and unipotent over $F$.
\end{proof}

\begin{lemma} \label{trans} Let $\spp\to S$ correspond to a finite field extension $K/F$ and let $G$ be a smooth, commutative, connected and quasi-projective $F$-group scheme. Then $R_{\e K/F}^{\e(1)}\lbe(G\e)$ is smooth and connected.
\end{lemma}
\begin{proof} Let $L$ denote the separable closure of $F$ in $K$. By Proposition \ref{nsm}, $N_{G_{\lbe L}\lbe, \e K\be/\lbe L}\colon R_{\e K/L}\lbe(G_{\be K})\to G_{\lbe L}$ is a smooth surjection of quasi-projective $L$-group schemes. Thus, by \cite[Corollary A.5.4(3), p.~507]{cgp} and the transitivity of Weil restrictions \cite[(2.41), p.~16]{gfr}, there exists a canonical exact sequence of smooth and quasi-projective $F$-group schemes
\begin{equation}\label{jai}
0\to R_{\e L/\lbe F}(\lbe R_{\e K\lbe/\lbe L}^{\e(1)}\lbe(G_{\lbe L}))\to R_{\e K/F}\lbe(G_{\be K})\to R_{\e L/F}(G_{\be L})\to 0,
\end{equation}
where the third map above can be identified with $R_{L\lbe/\lbe F}(N_{G_{\lbe L}\lbe, \e K\lbe/\lbe L})$. We now apply Lemma \ref{ker-cok} (in the abelian category $(\spec F\e)_{\fl}^{\le\sim}\e$) to the pair of $F$-morphisms
\begin{equation}\label{pai}
R_{\e K/F}\lbe(G_{\be K}\e)\twoheadrightarrow R_{\e L/F}(G_{\be L})\twoheadrightarrow G
\end{equation}
whose composition is $N_{G,\e L\lbe/\lbe F}\be\circ\be R_{L\lbe/\lbe F}(N_{G_{\lbe L}\lbe, \e K\lbe/\lbe L})=N_{G,\e K\lbe/\lbe F}$. Since the kernel of the first map in \eqref{pai} is $R_{\e L/\lbe F}(\lbe R_{\e K\lbe/\lbe L}^{\e(1)}\lbe(G_{\lbe L}))$ by the exactness of \eqref{jai}, we obtain an exact sequence of smooth and quasi-projective $F$-group schemes
\begin{equation}\label{per}
0\to R_{\e L/\lbe F}(\lbe R_{\e K\lbe/\lbe L}^{\e(1)}\lbe(G_{\lbe L}))\to R_{\e K/F}^{\e(1)}\lbe(G\e)\to R_{L/F}^{\e(1)}\lbe(G\e)\to 0.
\end{equation}
By Lemma \ref{rad}, $R_{\e K/L}^{\e(1)}\lbe(G_{\lbe L}\e)$ is smooth, connected and unipotent over $L$, whence $R_{\e L/\lbe F}(\lbe R_{\e K\lbe/\lbe L}^{\e(1)}\lbe(G_{\lbe L}))$ is smooth, connected and unipotent over $F$ by \cite[Proposition A.5.9, p.~514]{cgp} and \cite[Proposition A.3.7, p.~84]{oes}. On the other hand, $R_{L/F}^{\e(1)}\lbe(G\e)$ is connected since it is a form of $G^{\le\lle [L\colon \! F\e]-1}$ by Proposition \ref{wrp2}. The lemma now follows from the exactness of \eqref{per} and \cite[Lemma 2.55]{gfr}.	
\end{proof}

\begin{proposition} \label{wrp0}  Let $\spp\to S$ correspond to $B/F$, where $F$ is a field and $B$ is a local and finite $F$-algebra. Let $G$ be a smooth, commutative, connected and quasi-projective $F$-group scheme. Then $R_{\e B/F}^{\e(1)}\lbe(G\e)$ is smooth and connected.	
\end{proposition}
\begin{proof} Let $K$ be the residue field of $B$. As in the proof of Lemma \ref{trans}, there exists a canonical exact sequence of smooth and quasi-projective $F$-group schemes
\[
0\to R_{\e K\be/\lbe F}(\lbe R_{\e B\lbe/\lbe K}^{\e(1)}\lbe(G_{\be K}))\to R_{\e B/F}^{\e(1)}\lbe(G\e)\to R_{\e K/F}^{\e(1)}\lbe(G\e)\to 0,
\]
where the right-hand group above is connected by Lemma \ref{trans}. Now, by \cite[Proposition 20.2]{gfr}, $R_{\e B/K}^{\e(1)}\lbe(G_{K})$ is smooth, connected and unipotent over $K$. The rest of the proof is similar to the last part of the proof of Lemma \ref{trans}.
\end{proof}

\begin{proposition}\label{con} Let $\spp\to S$ be a finite and faithfully flat morphism of locally noetherian schemes and let $G$ be a smooth, commutative and quasi-projective $S$-group scheme with connected fibers. Then $R_{\e S^{\lle\prime}\!/\lbe S}\lbe(G\e)$ and $R_{\e S^{\lle\prime}\!/\lbe S}^{\e(1)}\lbe(G\e)$ are smooth and commutative with connected fibers.	
\end{proposition}
\begin{proof} By Lemma \ref{bic}, we need only consider $R_{\e S^{\lle\prime}\!/\lbe S}^{\e(1)}\lbe(G\e)$. Let $s$ be a point of $S$ with residue field $k(s)$ and write $\spp\times_{S}\spec k(s)=\spec B(s)$. By \cite[Exercise 3, p.~92 and Theorem 8.7, p.~90]{am}, there exists an isomorphism of $k(s)$-algebras $B(s)\simeq \prod_{\le i=1}^{\le m}\be B(s)_{i}$, where each $B(s)_{i}$ is a local and finite $k(s)$-algebra. Thus, by Proposition \ref{wtf}, there exist isomorphisms of $k(s)$-group schemes 
\[
R_{\e S^{\lle\prime}\!/\lbe S}^{\e(1)}\lbe(G\e)_{s}\simeq R_{B(\lbe s)\lbe/ k(s)}^{\e(1)}\lbe(G_{\be\lle s})\simeq \prod_{i=1}^{m}R_{B(\lbe s)_{i}\lbe/ k(s)}^{\e(1)}\lbe(G_{\be\lle s}).
\]	
The proposition is now immediate from Proposition \ref{wrp0} and \cite[${\rm VI_{A}}$, Proposition 2.1.1]{sga3}.
\end{proof}

\begin{definition}\label{adm} The pair $(\e f,G\e)$ is called {\it admissible} if \begin{enumerate}
\item[(i)] $f\colon \spp\to S$ is a finite and faithfully flat morphism of locally noetherian schemes of constant rank $n\geq 2$,
\item[(ii)] $G$ is a smooth, commutative and quasi-projective $S$-group scheme with connected fibers, and
\item[(iii)] for every point $s\in S$ such that ${\rm char}\, k(\lbe s\lbe)$ divides $n$,
\begin{enumerate}
\item[(iii.1)] $G_{k(\lbe s\lbe)}$ is a semiabelian $k(\lbe s\lbe)$-variety, and
\item[(iii.2)] $f_{\lbe s}\colon \spp\times_{S}\spec k(\lbe s\lbe)\to \spec k(\lbe s\lbe)$ is \'etale.
\end{enumerate}
\end{enumerate}
\end{definition}

\begin{proposition}\label{isg2} Assume that $(\e f, G\e)=(\spp\!/S, G\e)$ is admissible (see Definition {\rm \ref{adm}}) and let $n\geq 2$ be the rank of $f$.
If $H=G,\e R_{\e S^{\lle\prime}\be/ S}(G_{\lbe S^{\lle\prime}}\be)$ or $R_{\e S^{\lle\prime}\!/\lbe S}^{\e(1)}\lbe(G\e)$, then $n\colon H\to H$ is faithfully flat and locally of finite presentation.
\end{proposition}
\begin{proof} It was shown in \cite[Proposition 3.3]{ga18} that conditions (ii) (minus the quasi-projectivity hypothesis) and (iii.1) of Definition \ref{adm} imply the proposition when $H=G$. Thus, to  establish the proposition when $H=R_{\e S^{\lle\prime}\be/ S}(G_{\lbe S^{\lle\prime}}\be)$ or $H=R_{\e S^{\lle\prime}\!/\lbe S}^{\e(1)}\lbe(G\e)$, it suffices to check that
conditions (ii) (minus the quasi-projectivity hypothesis) and (iii.1) of Definition \ref{adm} hold true when $G$ is replaced with $R_{\le S^{\prime}\!/ S}(G_{\lbe S^{\prime}})$ or $R_{\e S^{\lle\prime}\!/\lbe S}^{\e(1)}\lbe(G\e)$. For condition (ii) (minus the quasi-projectivity hypothesis), see Proposition \ref{con}. Now let $s\in S$ be such that ${\rm char}\, k(\lbe s\lbe)$ divides $n$. Then $R_{\e S^{\lle\prime}\!/\lbe S}\lbe(G\e)_{s}\simeq R_{B(\lbe s)\lbe/ k(s)}\lbe(G_{\be\lle s})$ and similarly for $R_{\e S^{\lle\prime}\!/\lbe S}^{\e(1)}\lbe(G\e)$, where $B(s)$ is a finite and \'etale $k(s)$-algebra by Definition \ref{adm}(iii.2). Since $G_{\be\lle s}$ is a semiabelian $k(\lbe s\lbe)$-variety by Definition \ref{adm}(iii.1), $R_{B(\lbe s)\lbe/ k(s)}\lbe(G_{\be\lle s})$ and $R_{B(\lbe s)\lbe/ k(s)}^{\e(1)}\lbe(G_{\be\lle s})$
are semiabelian $k(\lbe s\lbe)$-varieties by Proposition \ref{radm}, i.e., $R_{\le S^{\prime}\!/ S}(G_{\lbe S^{\prime}})$ and $R_{\e S^{\lle\prime}\!/\lbe S}^{\e(1)}\lbe(G\e)$ satisfy Definition \ref{adm}(iii.1). The proof is now complete.
\end{proof}

\section{Proof of the main theorem}\label{mid}

Let $f\colon \spp\to S$ be a finite and faithfully flat morphism of locally noetherian schemes of constant rank $n\geq 2$ and let $G$ be a commutative and quasi-projective $S$-group scheme.

The Cartan-Leray spectral sequence associated to $(\e f, G\e)$ relative to the $\tau$ topology
\[
H^{\le r}\be(S_{\tau},R^{\e s}\be f_{\lbe *}\lbe(G_{\be S^{\lle\prime}}\be))\implies H^{\e r+s}\lbe(\spp_{\tau},G\e)
\]
induces edge morphisms $e_{\tau}^{\le (r)}\colon H^{\le r}\be(S_{\tau},R_{\e S^{\lle\prime}\be/ S}(G_{\lbe S^{\lle\prime}}\be)\e)\to H^{\le r}\be(\spp_{\tau},G\e)$ for every $r\geq 0$ \cite[Proposition 2.3.1, p.~14]{t}. The map $e_{\tau}^{\le (0)}$ is the isomorphism \eqref{wr} for $X^{\prime}=G_{\lbe S^{\lle\prime}}$. Further, by \cite[Theorem 6.4.2(ii), p.~128]{t}, the maps $e_{\et}^{\lle (r)}$ are isomorphisms for every $r\geq 0$. Note, however, that the maps $e_{\fl}^{\lle (r)}$ are not isomorphisms in general \cite[XXIV, Remarks 8.5]{sga3}.

Now consider
\[
j^{(r)}_{G,\e \tau}=H^{r}\be(S_{\tau}\lbe,\e j_{\e G,\e S^{\lle\prime}\be/ S}\e)\colon H^{\le r}\be(S_{\tau},G\e)\to H^{\le r}\be(S_{\tau},R_{\e S^{\lle\prime}\be/ S}(G_{\lbe S^{\lle\prime}}\be)\e),
\]
where $j_{\e G,\e S^{\lle\prime}\be/ S}$ is the map \eqref{imor} associated to $X=G$. The composition 
\[
H^{\e r}\lbe(S_{\et},G\e)\overset{\! j^{(r)}_{G,\lle \et}}{\lra} H^{\e r}\lbe(S_{\et},R_{\e S^{\lle\prime}\be/ S}(G_{\lbe S^{\lle\prime}}\be)\e)\underset{\sim}{\overset{\!e_{\et}^{(r)}}{\lra}} H^{\e r}\lbe(\spp_{\et},G\e)
\]
is the natural $r$-th restriction map in \'etale cohomology
\begin{equation}\label{res}
{\rm Res}_{\e G}^{(r)}\colon H^{\le r}\lbe(S_{\et},G\e)\to H^{\le r}\lbe(\spp_{\et},G\e).
\end{equation}

We now consider
\[
N_{G,\e\tau}^{(\e r\le)}=H^{\le r}(S_{\tau},N_{G,\e S^{\lle\prime}\be/ S})\colon H^{\le r}\be(S_{\tau},R_{\e S^{\lle\prime}\be/ S}(G_{\lbe S^{\lle\prime}}\be)\e)\to H^{\e r}\lbe(S_{\tau},G\e),
\]
where $N_{G,\e S^{\lle\prime}\be/ S}$ is the norm morphism \eqref{tmor}.
By \cite[XVII, Example 6.3.18]{sga4}, $N_{G,\e\tau}^{(\e 0\le)}$ can be identified with a map
\begin{equation}\label{n0}
N_{S^{\lle\prime}\!\lbe/S}\colon G(\spp\le)\to G(S\e)
\end{equation}
that agrees with the usual norm map if $G=\bg_{m,\le S}$. The composition
\begin{equation}\label{cores}
{\rm Cores}_{\e G}^{(r)}\colon H^{\e r}\lbe(\spp_{\et},G\e)\underset{\!\sim}{\overset{\!(e_{\et}^{\le (r)})^{-1}}{\lra}}H^{\e r}\lbe(S_{\et},R_{\e S^{\lle\prime}\be/ S}(G_{\lbe S^{\lle\prime}}\be)\e)\overset{\!N_{G,\le\et}^{(\e r\le)}}{\lra}H^{\le r}\lbe(S_{\et},G\e)
\end{equation}
is the $r$-th corestriction map in \'etale cohomology. For $r=0$ the maps \eqref{n0} and \eqref{cores} will be identified, i.e.,
\begin{equation}\label{ide}
{\rm Cores}_{\e G}^{(0)}=N_{S^{\lle\prime}\!\lbe/S}.
\end{equation}
Further, the maps ${\rm Cores}_{\e \bg_{m,S}}^{(i)}$ for $i=1$ and $2$ will both be denoted by ${\rm Cores}_{S^{\lle\prime}\!/S}$ since no ambiguity will result from this choice of notation. If $n\geq 1$ is an integer, we will write ${\rm Res}_{\e G,\e n}^{(r)}$ and ${\rm Cores}_{\e G,\e n}^{(r)}$ for $({\rm Res}_{\e G}^{(r)}\le)_{n}$ and $({\rm Cores}_{\e G}^{(r)}\le)_{n}$, respectively. The following equalities hold for all $r\geq 0$ and all $n\geq 1$:
\begin{eqnarray}
\krn\e j^{(r)}_{G,\e \et}&=&\krn\e {\rm Res}_{\e G}^{(r)}\label{note1}\\
\krn\big(\e j^{(r)}_{G,\e \et}/n\big)&=&\krn\big({\rm Res}_{\e G}^{(r)}/n)\label{note2}\\
\cok\big((N_{G,\e\et}^{(\e r\le)}\le)_{n}\big)&=&\cok\e{\rm Cores}_{\e G,\e n}^{(r)}\label{note4}\\
\cok\big((\e j^{(r)}_{G,\e \et}\le)_{n}\big)&=&\cok\e{\rm Res}_{\e G,\e n}^{(r)}\label{note5}\\
\krn\le (N_{G,\e\et}^{(\e r\le)}\le)_{n}&=&\krn\e{\rm Cores}_{\e G,\e n}^{(r)}.\label{note6}
\end{eqnarray}

Next, by \eqref{nis}, the composition
\begin{equation}\label{nce}
H^{\e r}\lbe(S_{\tau},G\e)\overset{\! j^{(r)}_{G,\lle \tau}}{\lra} H^{\e r}\lbe(S_{\tau},R_{\e S^{\lle\prime}\be/ S}(G_{\lbe S^{\lle\prime}}\be)\e)\overset{\!N_{G,\e\tau}^{(\e r\le)}}{\lra} H^{\e r}\lbe(S_{\tau},G\e)
\end{equation}
is the multiplication by $n$ map on $H^{\le r}\lbe(S_{\tau},G\e)$. Thus $\krn\e j_{\le G,\e\tau}^{\le(r)}$ and $\cok\e N_{G,\e\tau}^{(\e r\le)}$ are $n$-torsion abelian groups and \eqref{co1} yields
\begin{equation}\label{ntor}
\krn\lle (\e j_{ G,\e\tau}^{\le(r)})_{n}=\krn\e j_{\le G,\e\tau}^{\le(r)}
\end{equation}
and 
\begin{equation}\label{cotor}
\cok\le(N_{G,\e\tau}^{(\e r\le)}/n)=\cok N_{G,\e\tau}^{(\e r\le)}.
\end{equation}
Now observe that \eqref{nce} induces three complexes (in degrees $0,1$ and $2$) of $n$-torsion abelian groups, namely
\begin{equation}\label{comp1}
C^{\e\bullet}_{\be/n}\be(\lle r,\lle G\e)=\big(H^{\le r}\lbe(S_{\et},G\e)/n\overset{\!j^{(r)}_{G,\lle \et}\e/n}{\lra} H^{\e r}\lbe(S_{\et},R_{\e S^{\lle\prime}\be/ S}(G_{\lbe S^{\lle\prime}}\be)\e)/n\overset{\!N_{G,\e\et}^{(\e r\le)}/n}{\lra}  H^{\le r}\lbe(S_{\et},G\e)/n\le\big) ,
\end{equation}
\begin{equation}\label{comp2}
C^{\e\bullet}\lbe(\lle r,\lle G_{\lbe n})=\big(H^{\le r}\lbe(S_{\fl},G_{\lbe n}\e)\overset{\!j_{G_{\lbe n}\lbe,\lle\fl}^{(r)}}{\lra} H^{\e r}\lbe(S_{\fl},R_{\e S^{\lle\prime}\be/ S}(G_{\lbe n,\e S^{\lle\prime}}\be)\e)\overset{\!N_{G_{\lbe n}\lbe,\lle\fl}^{(\e r\le)}}{\lra}  H^{\le r}\lbe(S_{\fl},G_{\lbe n}\e)\big),
\end{equation}
and
\begin{equation}\label{comp3}
C^{\e\bullet}_{\be n}\be(\lle r,\lle G\e)=\big(H^{\le r}\lbe(S_{\et},G\e)_{n}\overset{\!\big(\e j_{ G,\e\et}^{\le(r)}\big)_{\! n}}{\lra} H^{\e r}\lbe(S_{\et},R_{\e S^{\lle\prime}\be/ S}(G_{\lbe S^{\lle\prime}}\be)\e)_{n}\overset{\!\big( N_{G,\e\et}^{(\e r\le)}\big)_{\be n}}{\lra}  H^{\le r}\lbe(S_{\et},G\e)_{n}\big).
\end{equation}
Using \eqref{note1}-\eqref{note4} and \eqref{ntor}-\eqref{cotor}, we have 
\begin{eqnarray}
H^{\le 0}\lbe(C^{\e\bullet}_{\be/n}\be(\lle r,\lle G\e))&=&\krn\big({\rm Res}_{\e G}^{(r)}/n)\label{ein}\\
H^{\le 0}\lbe(C^{\e\bullet}\be(\lle r,\lle G_{\lbe n}))&=&\krn\e j_{G_{\lbe n}\lbe,\lle\fl}^{(r)}\label{zwei}\\
H^{\le 0}\lbe(C^{\e\bullet}_{\be n}\be(\lle r,\lle G\e))&=&\krn\e {\rm Res}_{\e G}^{(r)}\label{drei}\\
H^{\le 2}\lbe(C^{\e\bullet}_{\be/n}\be(\lle r,\lle G\e))&=&\cok\e{\rm Cores}_{\e G}^{(r)}\label{ein2}\\
H^{\le 2}\lbe(C^{\e\bullet}\be(\lle r,\lle G_{\lbe n}))&=&\cok\e N_{G_{\lbe n}\lbe,\lle\fl}^{(\e r\le)}\label{zwei2}\\
H^{\le 2}\lbe(C^{\e\bullet}_{\be n}\be(\lle r,\lle G\e))&=&\cok\e{\rm Cores}_{\e G,\e n}^{(r)}.\label{drei2}
\end{eqnarray}
We now define
\begin{equation}\label{psiN}
\Psi_{\be N}(n,G\e)=\{x\in G(\spp\le)\colon N_{S^{\lle\prime}\!\lbe/S}(x)\in G(S\e)^{n}\},
\end{equation}
where $N_{S^{\lle\prime}\!\lbe/S}$ is the map \eqref{n0}. Then $\Psi_{\be N}(\e n,G\e)$ is a subgroup of $ G(\spp\le)$ which contains $G(S\e)\e G(\spp)^{n}$ and we have
\begin{equation}\label{four}
H^{\le 1}\lbe(C^{\e\bullet}_{\be/n}\be(\lle 0,\lle G\e))=\frac{\Psi_{\be N}( n,G\e)}{G(S\e)\e G(\spp\le)^{n}}.
\end{equation}

\smallskip

We now assume that $(\e f,G\e)$ is an admissible pair (see Definition \ref{adm}). 

\smallskip

Since $(\e f,G\e)$ is admissible, Proposition \ref{isg2} yields an exact and commutative diagram in $\sfs$
\begin{equation}\label{xcool}
\xymatrix{0\ar[r]& R_{\e S^{\lle\prime}\be/ S}^{\e(1)}(G_{\lbe n})\ar@{^{(}->}[d]^{a}\ar[r]& R_{\e S^{\lle\prime}\be/ S}^{\e(1)}(G\e)\ar@{^{(}->}[d]\ar[r]^{n}& R_{\e S^{\lle\prime}\be/ S}^{\e(1)}(G\e)\ar@{^{(}->}[d]\ar[r]&0\\
0\ar[r]&\ar[d]^{N_{G_{n}\lbe,\le S^{\lle\prime}\!/ S}} R_{\e S^{\lle\prime}\be/ S}(G_{\lbe n,\e S^{\lle\prime}}\be)\ar[r]& \ar@{->>}[d]^{N_{G\lbe,\le S^{\lle\prime}\!/ S}} R_{\e S^{\lle\prime}\be/ S}(G_{\lbe S^{\lle\prime}}\be)\ar[r]^{n}&\ar@{->>}[d]^{N_{G\lbe,\le S^{\lle\prime}\!/ S}} R_{\e S^{\lle\prime}\be/ S}(G_{\lbe S^{\lle\prime}}\be)\ar[r]&0\\
0\ar[r]&G_{\lbe n}\ar[r]& G\ar[r]^{n}&G\ar[r]&0,
}
\end{equation}
where $a$ is the inclusion morphism. A diagram chase (or an application of the snake lemma to the bottom half of the above diagram) shows that $N_{G_{n}\lbe,\le S^{\prime}\!/ S}$ is surjective, whence
\begin{equation}\label{blh}
0\to R_{\e S^{\lle\prime}\be/ S}^{\e(1)}(G_{\lbe n})\overset{a}{\lra} R_{\e S^{\lle\prime}\be/ S}(G_{\lbe n,\e S^{\lle\prime}}\be)\overset{N_{\lbe G_{n}\lbe,\lle S^{\prime}\!\lbe/\be S}}{\lra} G_{\lbe n}\to 0
\end{equation}
is an exact sequence in $\sfs$. On the other hand, by \eqref{nis} applied to $G_{\lbe n}$,  
\[
G_{\lbe n}\e\overset{j_{\le G_{\lbe n}\lbe,\e S^{\lle\prime}\!/ S}}{\hookrightarrow } R_{\e S^{\lle\prime}\be/ S}(G_{\lbe n,\e S^{\lle\prime}}\be)\overset{\!N_{G_{n}\lbe,\le S^{\lle\prime}\!/ S}}{\lra} G_{\lbe n}
\]
is a complex of commutative and quasi-projective $S$-group schemes. Thus there exists a closed immersion of commutative $S$-group schemes
\begin{equation}\label{bs}
b\colon G_{\lbe n}\to R_{\e S^{\lle\prime}\be/ S}^{\e(1)}(G_{\lbe n}).
\end{equation}
We will write
\begin{equation}\label{gen}
G(n)=R_{\e S^{\lle\prime}\be/ S}^{\e(1)}(G_{\lbe n})/G_{\lbe n}
\end{equation}
for the cokernel of $b$ in $\sfs$. Note that $G(n)$ is represented by an $S$-group scheme if $G_{\lbe n}$ is flat over $S$ and $\dim S\leq 1$ \cite[Theorem 4.C, p.~53]{an}. 
Now, by definition of $G(n)$, there exists a canonical exact sequence in $\sfs$
\begin{equation}\label{lhb}
0\to G_{\lbe n}\overset{b}{\lra} R_{\e S^{\lle\prime}\be/ S}^{\e(1)}(G_{\lbe n})\overset{q}{\lra} G(n)\to 0,
\end{equation}
where $q$ is the projection and $b$ is the inclusion map. Note that $a\circ b=j_{\le G_{n}\lbe,\e S^{\lle\prime}\!/ S}$. Now, for every $r\geq 0$,
consider
\begin{eqnarray*}
a^{(r)}&=& H^{\le r}\be(S_{\fl},a\le)\colon H^{\le r}\be(S_{\fl},R_{\e S^{\lle\prime}\be/ S}^{\e(1)}(G_{\lbe n}))\to H^{\le r}\be(S_{\fl},R_{\e S^{\lle\prime}\be/ S}(G_{\lbe n,\e S^{\lle\prime}}\be)),\\
b^{(r)}&=&H^{\le r}\be(S_{\fl},b\le)\colon H^{\le r}\be(S_{\fl},G_{\lbe n})\to 
H^{\le r}\be(S_{\fl},R_{\e S^{\lle\prime}\be/ S}^{\e(1)}(G_{\lbe n})),\\
q^{(r)}&=& H^{\le r}\be(S_{\fl},q\le)\colon H^{\le r}\be(S_{\fl},R_{\e S^{\lle\prime}\be/ S}^{\e(1)}(G_{\lbe n}))\to H^{\le r}\be(S_{\fl},G(n)).
\end{eqnarray*}
Then the sequences \eqref{blh} and \eqref{lhb} induce exact sequences of abelian groups
\begin{equation}\label{pin}
\dots\to H^{\le r-1}\be(S_{\fl},R_{\e S^{\lle\prime}\be/ S}(G_{\lbe n,\e S^{\lle\prime}}\be))\overset{N_{\lbe G_{\lbe n}\lbe,\lle\fl}^{(r-1)}}{\lra}
H^{\le r-1}\lbe(S_{\fl},G_{\lbe n})\overset{\delta^{\le(r-1)}}{\lra} H^{\le r}\be(S_{\fl},R_{\e S^{\lle\prime}\be/ S}^{\e(1)}(G_{\lbe n}))\overset{a^{(r)}}{\lra}\dots 
\end{equation}
and 
\begin{equation}\label{nip}
\begin{array}{rcl}
\dots\to H^{\le r}\be(S_{\fl},G_{\lbe n}\be)\overset{b^{(r)}}{\lra}
H^{\le r}\be(S_{\fl},R_{\e S^{\lle\prime}\be/ S}^{\e(1)}(G_{\lbe n}))&\overset{q^{(r)}}{\lra}& H^{\le r}\be(S_{\fl},G(n))\\
&\overset{\partial^{\e(r)}}{\lra}&H^{\le r+1}\lbe(S_{\fl},G_{\lbe n})\to\dots,
\end{array}
\end{equation}
where the maps $\delta^{\e(r)}\!$ and $\partial^{\e(r)}$ are connecting morphisms induced by \eqref{blh} and \eqref{lhb}, respectively. Note that, since $a^{(r)}\circ b^{(r)}=j_{G_{\lbe n}\lbe,\lle\fl}^{(r)}$ and $\krn\e q^{(r)}=\img\e b^{(r)}$, we have
\begin{equation}\label{surj2}
a^{(r)}(\e\krn\e q^{(r)})=\img j_{G_{\lbe n}\lbe,\lle\fl}^{(r)}.
\end{equation}
Now, for every $r\geq 1$, let
\begin{equation}\label{ar}
\gamma^{(r)}\colon H^{\le r-1}\lbe(S_{\fl},G_{\lbe n})\to  H^{\le r}\be(S_{\fl},G(n))
\end{equation}
be the composition
\[
H^{\le r-1}\be(S_{\fl},G_{\lbe n})\overset{\delta^{\e(r-1)}}{\lra} H^{\le r}\be(S_{\fl},R_{\e S^{\lle\prime}\be/ S}^{\e(1)}(G_{\lbe n}))\overset{q^{(r)}}{\lra} H^{\le r}\be(S_{\fl},G(n)).
\]
By the exactness of \eqref{pin} and \eqref{nip}, there exists a canonical isomorphism of abelian groups
\begin{equation}\label{tri}
\krn\e a^{(r)}\be/\, \krn\e a^{(r)}\cap \krn\e q^{(r)}\isoto q^{(r)}(\krn\e a^{(r)})=q^{(r)}(\img\e \delta^{\e(r)})=\img\e \gamma^{\e (r)}.
\end{equation}
Next consider the complex
\begin{equation}\label{gam}
\varGamma^{\e\bullet}_{\!n}\lbe(r, G\e)=(H^{\le r-1}\be(S_{\fl},G_{\lbe n})\overset{\gamma^{(r)}}{\lra}H^{\le r}\be(S_{\fl},G(n))\overset{\partial^{\e (r)}}{\lra}H^{\le r+1}\be(S_{\fl}, G_{\lbe n})).
\end{equation}

\begin{lemma}\label{ok} For every $r\geq 1$, there exists a canonical isomorphism of $n$-torsion abelian groups
\[
H^{\le 1}\lbe (C^{\e\bullet}\lbe(\lle r,\lle G_{\lbe n}))\simeq H^{\le 1}\be (\varGamma^{\e\bullet}_{\!n}\lbe(r, G\e)),
\]
where the complexes $C^{\e\bullet}\be(\lle r,\lle G_{\lbe n})$ and $\varGamma^{\e\bullet}_{\!n}\lbe(r, G\e)$ are given by \eqref{comp2} and \eqref{gam}, respectively. 
\end{lemma}
\begin{proof} By the exactness of \eqref{nip} and the identity \eqref{surj2}, the following diagram is exact and commutative
\[
\xymatrix{0\ar[r]&\ar@{^{(}->}[d]\krn\e a^{(r)}\cap \krn\e q^{(r)}\ar[r]&\krn\e q^{(r)}\ar@{^{(}->}[d]\ar[r]^(.45){a^{(r)}}&\img j_{G_{\lbe n}\lbe,\lle\fl}^{(r)}\ar@{^{(}->}[d]\ar[r]&0\\
0\ar[r]&\krn\e a^{(r)}\ar[r]&H^{\le r}\be(S_{\fl},R_{ S^{\lle\prime}\be/ S}^{\e(1)}(G_{\lbe n}))\ar[r]^(.6){a^{(r)}}&\krn\e N_{\be G_{\lbe n}\lbe,\lle\fl}^{(r)}\ar[r]&0.
}
\]
The above diagram induces the top row of the following commutative diagram of $n$-torsion abelian groups with exact rows
\[
\xymatrix{0\ar[r]&\ar[d]^{\sim}\frac{\krn\e a^{(r)}}{\krn\e a^{(r)}\e\cap\e \krn\e q^{(r)}}\ar[r]&\frac{H^{\le r}\be(S_{\le\fl},\e R_{S^{\lle\prime}\be/ S}^{\e(1)}(G_{\lbe n}))}{\krn\e q^{(r)}}\ar[d]^{\sim}\ar[r]& H^{\le 1}\lbe (C^{\e\bullet}\lbe(\lle r,\lle G_{\lbe n}))\ar@{=}[d]\ar[r]&0\\
0\ar[r]& \ar@{=}[d]q^{(r)}(\krn\e a^{(r)})\ar[r]&\ar@{=}[d]\img\e q^{\e (r)}\ar[r]& H^{\le 1}\lbe (C^{\e\bullet}\lbe(\lle r,\lle G_{\lbe n}))\ar@{=}[d]\ar[r]&0\\
0\ar[r]& \img\e \gamma^{(r)}\ar[r]&\krn\e\partial^{\e (r)}\ar[r]& H^{\le 1}\lbe (C^{\e\bullet}\lbe(\lle r,\lle G_{\lbe n}))\ar[r]&0,
}
\]
where the left-hand vertical composition is the isomorphism \eqref{tri} and the top middle vertical isomorphism comes from \eqref{nip}. Thus we obtain a canonical isomorphism $\krn\e\partial^{\e (r)}\be/\, \img\e \gamma^{(r)}=H^{\le 1}\lbe (\varGamma^{\e\bullet}_{\!n}\lbe(r, G\e))\isoto H^{\le 1}\lbe (C^{\e\bullet}\lbe(\lle r,\lle G_{\lbe n}))$, as claimed.
\end{proof}

Next, the middle and bottom rows of diagram \eqref{xcool} induce short exact sequences of abelian groups for every integer $r\geq 1$
\[
H^{\le r-1}\lbe(S_{\et},R_{\e S^{\lle\prime}\be/ S}(G_{\lbe S^{\lle\prime}}\be)\e)/n\hookrightarrow H^{\le r}(S_{\fl},R_{\e S^{\lle\prime}\be/ S}(G_{\lbe n,\e S^{\lle\prime}}\be)\le)\twoheadrightarrow  H^{\le r}\lbe(S_{\et},R_{\e S^{\lle\prime}\be/ S}(G_{\lbe S^{\lle\prime}}\be)\e)_{n}
\]
and
\[
0\to H^{\le r-1}\lbe(S_{\et},G\e)/n\to H^{\le r}(S_{\fl}, G_{\lbe n}\le)\to  H^{\le r}\lbe(S_{\et},G\e)_{n}\to 0.
\]
The preceding sequences fit into the following exact and commutative diagram of abelian groups
\[
\xymatrix{H^{\le r-1}\lbe(S_{\et},G\e)/n\,\ar[d]^{j_{ G,\lle\et}^{(r-1)}/n}\ar@{^{(}->}[r]& H^{\le r}(S_{\fl}, G_{\lbe n}\le)\ar@{->>}[r]\ar[d]^{j_{\le G_{n}\lbe,\le\fl}^{(r)}}&H^{\le r}\lbe(S_{\et},G\e)_{n}\ar[d]^{\big(\e j_{\lbe G,\e\et}^{(r)}\big)_{\be n}}\\
H^{\le r-1}\lbe(S_{\et},R_{\e S^{\lle\prime}\be/ S}(G_{\lbe S^{\lle\prime}}\be)\e)/n\,\ar@{^{(}->}[r]\ar[d]^{N_{\lbe G,\lle\et}^{(r-1)}/n}& H^{\le r}(S_{\fl},R_{\e S^{\lle\prime}\be/ S}(G_{\lbe n,\e S^{\lle\prime}}\be)\le)\ar[d]^{N_{\lbe G_{\lbe n}\lbe,\le\fl}^{(r)}}\ar@{->>}[r]& H^{\le r}\lbe(S_{\et},R_{\e S^{\lle\prime}\be/ S}(G_{\lbe S^{\lle\prime}}\be)\e)_{n}\ar[d]^{\big(N_{\le G\lbe,\le\et}^{(r)}\big)_{\be n}}\\
H^{\le r-1}\lbe(S_{\et},G\e)/n\,\ar@{^{(}->}[r]& H^{\le r}(S_{\fl}, G_{\lbe n}\le)\ar@{->>}[r]& H^{\le r}\lbe(S_{\et},G\e)_{n},
}
\]
i.e., there exists a canonical exact sequence of complexes of abelian groups
\begin{equation}\label{nice}
0\to C^{\e\bullet}_{\be/n}\be(\lle r-1,\lle G\e)\to C^{\e\bullet}\be(\lle r,\lle G_{\lbe n})\to C^{\e\bullet}_{\be n}\be(\lle r,\lle G\e)\to 0,
\end{equation}
where the left-hand, middle and right-hand complexes are given by \eqref{comp1}, 
\eqref{comp2} and \eqref{comp3}, respectively. We can now state the main theorem of the paper.

\begin{theorem}\label{main} Assume that the pair $(\e f, G\e)$ is admissible (see Definition {\rm \ref{adm}}) and let $n\geq 2$ be the rank of $f$. For every integer $r\geq 1$, there exists a canonical exact sequence of $n$-torsion abelian groups
\[
\begin{array}{rcl}
0&\to&\krn\lbe\big({\rm Res}_{\e G}^{(r-1)}\!/n)\to \krn\e j_{\le G_{\lbe n}\lbe,\le\fl}^{(r)}\to \krn\e {\rm Res}_{\lle G}^{(r)}\\
&\to& H^{\le 1}\lbe ( C^{\e\bullet}_{\be/n}\be(\lle r-1,\lle G\e))\to H^{\le 1}\lbe (\varGamma^{\e\bullet}_{\!n}\lbe(r, G\e))\to H^{\le 1}\lbe (C^{\e\bullet}_{\lbe n}\be(\lle r,\lle G\e))\\
&\to& \cok\e{\rm Cores}_{\e G}^{(r-1)}\to \cok\e N_{\be G_{n}\lbe,\le\fl}^{\lle(r)}\to \cok\e{\rm Cores}_{\e G,\e n}^{(r)}\to 0,
\end{array}
\]
where the complexes $C^{\e\bullet}_{\be/n}\be(\lle r-1,\lle G\e)$, $\varGamma^{\e\bullet}_{\!n}\lbe(r, G\e)$ and $C^{\e\bullet}_{\lbe n}\be(\lle r,\lle G\e)$ are given by \eqref{comp1}, \eqref{gam} and \eqref{comp3}, respectively.
\end{theorem}
\begin{proof} The sequence of the theorem is derived from the cohomology sequence induced by \eqref{nice} using the identities \eqref{ein}-\eqref{drei2} and the isomorphism of Lemma \ref{ok}.
\end{proof}

\begin{corollary}\label{fir} Under the hypotheses of the theorem, there exists a canonical exact sequence of $n$-torsion abelian groups
\[
\begin{array}{rcl}
0&\to&G(S\e)\be\cap\be G(\spp\le)^{n}\be/G(S\e)^{n}\to \check{H}^{\le 1}(\spp\!/\be S, G_{\lbe n}\lbe)\to \krn\e {\rm Res}_{\lle G}^{(1)}\\
&\to&  \Psi_{\be N}(n,G\e)/G(S\e)\e G(\spp\le)^{n}\to H^{\le 1}\lbe (\varGamma^{\e\bullet}_{\!n}\lbe(1, G\e))\to H^{\le 1}\lbe (C^{\e\bullet}_{\be n}\lbe(\lle 1,\lle G\e))\\
&\to& G(S\e)/\be N_{\lbe S^{\lle\prime}\!/S}(\lbe G\lbe(\spp\le))\to \cok\e N_{\be G_{n}\lbe,\le\fl}^{\lle(1)}\to \cok\e{\rm Cores}_{\e G,\e n}^{(1)}\to 0,
\end{array}
\]
where $\check{H}^{\le 1}(\spp\!/\be S, G_{\lbe n}\lbe)$ is the first \v{C}ech cohomology group of $G_{\lbe n}$ relative to the fppf covering $\spp\!/\be S$ and $\Psi_{\be N}(n,G\e)=\{x\in G(\spp\le)\colon N_{S^{\lle\prime}\!\lbe/S}(x)\in G(S\e)^{n}\}$.
\end{corollary}
\begin{proof} Set $r=1$ in the theorem and use the identities \eqref{ide}, \eqref{psiN} and \eqref{four} together with the canonical isomorphism $\krn\e j_{\le G_{n}\lbe,\le\fl}^{(r)}\simeq \check{H}^{\le 1}(\spp\!/\be S, G_{\lbe n}\lbe)$ of \cite[Proposition 2.1]{ga18}.	
\end{proof}

\section{Quadratic Galois coverings}\label{quad}

Let $(\e f,G\e)$ be an admissible pair such that $f$ is \'etale and let $\Delta$ be a finite group of order $n\geq 2$ which acts on $\spp/S$ from the right. Then $f$ is called a {\it Galois covering with Galois group $\Delta$} if the canonical map
\begin{equation}\label{gal}
\coprod_{\e \delta\e\in\e \Delta}\be \spp\to \spp\!\times_{S}\!\spp, (x,\delta\e)\mapsto (x,x\delta\e),
\end{equation}
is an isomorphism of $S$-schemes. See  \cite[V, Proposition 2.6 and Definition  2.8]{sga1}. In this case $n={\rm rank}\e f$ and there exists a canonical isomorphism of abelian groups
\begin{equation}\label{gpc}
\check{H}^{\le 1}(\spp\be/\be S, G_{\lbe n})\simeq H^{\le 1}(\Delta, G(\spp\le)_{\lbe n}),
\end{equation}
where $G(\spp\le)_{\lbe n}$ is a left $\Delta$-module via the given right action of $\Delta$ on $\spp\!/\lbe S$. See \cite[III, Example 2.6, p.~99]{mi1}.

By \eqref{gal} and Corollary \ref{mb} (with $\sppp=\spp$ there), there exists a canonical isomorphism of $\spp$-schemes
\[
R_{\e S^{\lle\prime}\!/\lbe S}^{\e(1)}(G_{\lbe n})_{S^{\lle\prime}}\isoto  G_{\lbe n,\e S^{\lle\prime}}^{\e n-1}
\]
(see the proof of Lemma \ref{vnice}). Under the preceding isomorphism, the closed immersion
$b_{\le S^{\lle\prime}}\colon G_{\lbe n,\e S^{\lle\prime}}\hookrightarrow R_{\e S^{\lle\prime}\!/\lbe S}^{\e(1)}(G_{\lbe n})_{S^{\lle\prime}}$ induced by \eqref{bs} corresponds to the diagonal embedding $d^{\e (n)}_{S^{\lle\prime}}\colon G_{\lbe n,\e S^{\lle\prime}}\to G_{\lbe n,\e S^{\lle\prime}}^{\e n-1}$. Now, for every $n\geq 3$, let
\[
c^{\le (n)}\colon G_{\lbe n}^{\e n-1}\to G_{\lbe n}^{\e n-2}, (x_{1},\dots,x_{n-1})\mapsto (x_{1}x_{n-1}^{-1},\dots, x_{n-2}x_{n-1}^{-1}),
\]
and let $c^{\le (2)}\colon G_{\lbe 2}\to S$ be the structural morphism of $G_{\lbe 2}$. Then $c^{\le (n)}$ is a morphism of commutative $S$-group schemes and there exists a canonical exact sequence of commutative $S$-group schemes
\begin{equation}\label{awe}
0\to G_{\lbe n}\overset{\!d^{\e(n)}}{\lra} G_{\lbe n}^{\e n-1}\overset{\!c^{\lle (n)}}{\lra} G_{\be n}^{\e n-2}\to 0,
\end{equation}
where $d^{\e(n)}$ is the diagonal $S$-morphism. Now recall $G(n)\e\in\e{\rm Ob}\e\sfs$ \eqref{gen}. We define an isomorphism of $\spp$-group schemes $G(n)_{S^{\lle\prime}}\isoto G_{\be n,\e S^{\lle\prime}}^{\e n-2}$ by the commutativity of the diagram
\[
\xymatrix{0\ar[r]& G_{\lbe n,\e S^{\lle\prime}}\ar@{=}[d]\ar[r]^(.4){b_{ S^{\lle\prime}}}&R_{\e S^{\lle\prime}\!/\lbe S}^{\e(1)}(G_{\lbe n})_{S^{\lle\prime}}\ar[d]^{\sim}\ar[r]^(.55){q_{S^{\lle\prime}}}& G(n)_{S^{\lle\prime}}\ar[d]^{\sim}\ar[r]&0\\
0\ar[r]& G_{\lbe n,\e S^{\lle\prime}}\ar[r]^{d_{\lbe S^{\lle\prime}}^{\e(n)}}& G_{\lbe n,\e S^{\lle\prime}}^{\e n-1}\ar[r]^(.55){c_{\lbe S^{\lle\prime}}^{\le(n)}}& G_{\lbe n,\e S^{\lle\prime}}^{\e n-2}\ar[r]&0,
}
\]
where the top (respectively, bottom) row is induced by \eqref{lhb} (respectively, \eqref{awe}). We conclude that $G(2)=0$, whence $H^{\le 1}\lbe (\varGamma^{\e\bullet}_{\!2}\lbe(r, G\e))=0$ for every $r\geq 1$ \eqref{gam}. Thus the following statement is immediate from Theorem \ref{main} and Corollary \ref{fir} using \eqref{four} and \eqref{gpc} for $n=2$.

\begin{theorem} \label{last} Let $f\colon \spp\to S$ be a quadratic Galois covering of locally noetherian schemes with Galois group $\Delta$ and let $G$ be a smooth, commutative and quasi-projective $S$-group scheme with connected fibers. Assume that, for every point $s\in S$ such that ${\rm char}\, k(\lbe s\lbe)=2$, $G_{k(\lbe s\lbe)}$ is a semiabelian $k(\lbe s\lbe)$-variety. Then, for every $r\geq 1$, there exist canonical exact sequences of $2$-torsion abelian groups
\begin{equation}\label{npr}
0\to\krn\lbe\big({\rm Res}_{\e G}^{(r-1)}\!/2)\to \krn\e j_{G_{2}\lbe,\le\fl}^{(r)}\to \krn\e {\rm Res}_{\lle G}^{(r)}\to H^{\le 1}\lbe ( C^{\e\bullet}_{\be/2}\lbe(\lle r-1,\lle G\e))\to 0
\end{equation}
and
\begin{equation}\label{nr}
0\to H^{\le 1}\lbe (C^{\e\bullet}_{2}\lbe(\lle r,\lle G\e))\to \cok\e{\rm Cores}_{\e G}^{(r-1)}\to \cok\e N_{\be G_{\lbe 2}\lbe,\le\fl}^{\lle(r)}\to \cok\e{\rm Cores}_{\e G,\e 2}^{(r)}\to 0,
\end{equation}
where the complexes $C^{\e\bullet}_{\be/2}\be(\lle r-1,\lle G\e)$ and $C^{\e\bullet}_{ 2}\be(\lle r,\lle G\e)$ are given by \eqref{comp1} and \eqref{comp3}, respectively.
In particular, if $r=1$, then \eqref{npr} induces an exact sequence of $2$-torsion abelian groups
\begin{equation}\label{cool}
0\to \frac{G(S\e)\be\cap\be G(\spp\le)^{2}}{G(S\e)^{2}}\to H^{\le 1}(\Delta, G(\spp\le)_{\lbe 2}\lbe)\to \krn\e {\rm Res}_{\lle G}^{(1)}\to \frac{\Psi_{\be N}(2,G\e)}{G(S\e)\e G(\spp\le)^{2}}\to 0,
\end{equation}
where $\Psi_{\be N}(\e 2,G\e)=\{x\in G(\spp\le)\colon N_{S^{\lle\prime}\!\lbe/S}(x)\in G(S\e)^{2}\e\}$.	
\end{theorem}

The above theorem yields, in particular, the following ``lower bound" for the relative cohomological Brauer group of a quadratic Galois covering of locally noetherian schemes in terms of their Picard groups.

\begin{corollary}\label{relb} Let $f\colon \spp\to S$ be a quadratic Galois covering of locally noetherian schemes. Then there exists a canonical surjection of $2$-torsion abelian groups
\[
\br^{\e\prime}\lbe(\spp\!\lbe/\lbe S\e)\twoheadrightarrow \frac{\krn[\e \pic\spp\!/\lbe 2\overset{{\rm Cores}_{S^{\prime}\!\lbe/\lbe S}}{\lra}\pic S\be/\lbe 2\e]}{\img[\e\pic S\be/2\overset{\,\,{\rm Res}_{S^{\prime}\!\lbe/S}}{\lra} \pic\spp\!/2\e]}\,,
\]	
where $\br^{\le\prime}\be(\spp\!\lbe/\lbe S\e)=\krn\lbe[\e{\rm Res}_{S^{\lle \prime}\!\lbe/S}\colon\br^{\le\prime}S\to \br^{\le\prime}\spp\,]$ is the relative  cohomological Brauer group of $\spp$ over $S$.
\end{corollary}
\begin{proof} Set $r=2$ and $G=\bg_{m,\e S}$ in \eqref{npr} and note that, by \eqref{comp1},
\[
H^{\le 1}\lbe ( C^{\e\bullet}_{\be/2}\lbe(\lle 1,\lle \bg_{m,\e S}))=\frac{\krn[\e \pic\spp\!/\lbe 2\overset{{\rm Cores}_{S^{\prime}\!\lbe/\lbe S}}{\lra}\pic S\be/\lbe 2\e]}{\img[\e\pic S\be/2\overset{\,\,{\rm Res}_{S^{\prime}\!\lbe/S}}{\lra} \pic\spp\!/2\e]}.
\] 	
\end{proof}

\smallskip

\begin{remark} Assume that $S$ is a noetherian scheme which admits an ample invertible sheaf, and similarly for $\spp$. Then $(\br\e\spp\le)_{2}=(\br\e S\e)_{2}$ and similarly for $\spp$ by \eqref{br}. Thus, setting $r=2$ and $G=\bg_{m,\e S}$ in \eqref{nr}\lle, we recover the Knus-Parimala-Srinivas injection \cite{kps, ps}:
\[
\frac{\krn[\e(\br\e\spp\le)_{2}\overset{{\rm Cores}_{S^{\lle\prime}\!/S}}{\lra}(\br\e S\le)_{2}]}{\img[(\br\e S\le)_{2}\overset{\,\,{\rm Res}_{S^{\lle\prime}\!\lbe/S}}{\lra}(\br\e\spp\le)_{2}]}\,\hookrightarrow\, \frac{\pic S}{{\rm Cores}_{S^{\lle\prime}\!/S}(\pic \spp\le)}.
\]
\end{remark}

\section{Arithmetical applications}\label{new}

Recall that a global field is either a number field, i.e., a finite extension of $\Q$, or a global function field, i.e., the function field of a smooth, projective and irreducible algebraic curve over a finite field.

Let $K\be/\be F$ be a quadratic Galois extension of global fields with Galois group $\Delta$ and let $\si$ be a nonempty finite set of primes of $F$ containing the archimedean primes and the non-archimedean primes that ramify in $K$. Let $\Sigma_{K}$ be the set of primes of $K$ that lie above the primes in $\Sigma$, write $\mathcal O_{\lbe F,\e \Sigma}$ for the ring of $\Sigma$-integers of $F$ and let $\mathcal O_{\lbe K,\e \Sigma_{K}}$ be the ring of $\Sigma_{K}$-integers of $K$. If $S=\spec\mathcal O_{\lbe F,\e \Sigma}$ and $\spp=\spec \mathcal O_{\lbe K,\e \Sigma_{K}}$, then the canonical morphism $f\colon \spp\to S$ induced by the inclusion $\ofs\subset\mathcal O_{\lbe K,\e \Sigma_{K}}$ is a quadratic Galois covering. Now let $G_{\be F}$ be a smooth, commutative and connected $F$-group scheme of finite type which admits a N\'eron-Raynaud model over $S$ with semiabelian reduction at every point $s\in S$ such that ${\rm char}\, k(s)=2$. If $G$ denotes the identity component of the indicated model, then the pair $(f,G\le)$ satisfies the hypotheses of Theorem \ref{last}\e. In this section we will consider the following specific choices of $G_{\be F}$:
\begin{enumerate}
\item $G_{\be F}$ is an {\it invertible} $F$-torus, i.e., $G_{\be F}$ is isomorphic to a direct factor of a finite direct product of $F$-tori of the form $R_{\le L\lbe/\lbe F}(\bg_{m,\le L})$, where $L/F$ is a finite separable extension of $F$, with multiplicative (i.e., toric) reduction at every point $s\in S$ such that ${\rm char}\, k(s)=2$.
\item $G_{\be F}$ is an abelian variety over $F$ with abelian (i.e., good) reduction over  $S$, i.e., $G$ is an abelian scheme over $S$.
\end{enumerate}

\begin{remark}\label{nik} In Case (2) above, $G(S\e)=G(F\e)$ and $G(\spp\e)=G(K\le)$ since $G$ and $G_{\be\spp}$ are N\'eron models of $G_{\be F}$ and $G_{\be K}=G_{\be F}\!\times_{\be F}\!K$, respectively \cite[\S1.2, Proposition 8, p.~15]{blr}.
\end{remark}

\smallskip

If $G_{\be F}$ is a smooth, commutative and connected $F$-group scheme of finite type  which admits a N\'eron-Raynaud model over $S$ with identity component $G$, the {\it N\'eron-Raynaud $\Sigma$-class group of $G_{\be F}$}, introduced in \cite[\S3]{ga12}, is the quotient
\begin{equation}\label{clgp}
C_{G_{\be F}\lbe,\e F,\e \Sigma}=G(\mathbb A_{S})/G(F)\e G(\lbe\mathbb A_{S}\lbe (S\e)),
\end{equation}
where $\mathbb A_{S}$ (respectively, $\mathbb A_{S}\lbe (S\le)$) is the ring of adeles (respectively, integral adeles) of $S$. The above group is known to be finite if $G_{\be F}$ is affine \cite[\S1.3]{c} (and thus for every $G_{\be F}$ in Case (1) above) and coincides with the $\Sigma$-ideal class group of $F$ when $G_{\be F}=\bg_{m,\le F}$. In general, \eqref{clgp} is a quotient of $\bigoplus_{s\in S}\lbe\Phi_{s}\lbe(G_{\be F}\be)(k(s))$, where $\Phi_{s}(G_{\be F}\be)$ is the $k(s)$-group scheme of connected components of the special fiber of the N\'eron-Raynaud model of $G_{\be F_{\lbe v}}:=G_{\be F}\times_{\lbe F} F_{\be v}$, where $v\notin\si$ is the prime of $F$ that corresponds to $s\in S$ \cite[Theorem 3.2]{ga12}. Consequently, \eqref{clgp} is also finite if $G_{\be F}$ is an abelian variety over $F$. Further, \eqref{clgp} is {\it trivial} in Case (2) above. The {\it capitulation problem for $C_{G_{\be F}\lbe,\e F,\e \Sigma}$} (relative to $K\be/F\e$) is the problem of describing the kernel of the map $j_{G_{\be F},\e K\be/\lbe F,\e\Sigma}\colon C_{G_{\be F}\lbe,\e F,\e \Sigma}\to C_{G_{\be K}\lbe,\e K,\e \Sigma_{K}}$ induced by $f$. This problem has a long history in the classical case $G_{\be F}=\bg_{m,\e F}$ when $F$ is a number field (see, e.g., \cite{st}) and was discussed in \cite[\S4]{ga10} for arbitrary finite Galois extension of global fields $K/F$ and arbitrary $F$-tori $G_{\be F}$ with multiplicative reduction over $S$.

\smallskip

Next, the {\it $\Sigma$-Tate-Shafarevich group of $G_{\be F}$} is the group 
\[
\sha^{1}_{\Sigma}(F,G\e)=\krn\!\be\left[H^{1}\be(F,G\e)\to \prod_{v\notin\Sigma}H^{1}\be(F_{v},G\e)\right],
\]
where $H^{1}\be(F,G\e)=H^{1}\be(F,G_{\be F}\be)$ and $H^{1}\be(F_{\be v},G\e)=H^{1}\be(F_{\be v},G_{\be F_{\lbe v}}\be)$ are Galois cohomology groups. For every prime $v$ of $F$, fix a prime $v^{\e\prime}$ of $K$ lying above $v$ and set $\Delta_{\e v^{\le\prime}}={\rm Gal}(K_{v^{\le\prime}}\be/\lbe F_{\be v})$. Every prime $w$ of $K$ lying above $v$ has the form $\sigma v^{\le\prime}$ for some $\sigma\in\Delta$ and $\sigma$ induces an isomorphism of abelian groups $H^{1}\be(K_{ w},G\e)\isoto H^{1}\be(K_{ v^{\le\prime}},G\e)$. It follows that the $\Sigma_{K}$-Tate-Shafarevich group of $G_{\be K}:=G_{\be F}\!\times_{\lbe F}\! K$ equals
\[
\sha^{1}_{\Sigma}\lbe(K,G\e):=\krn\!\be\left[H^{1}\be(K,G\e)\to \prod_{v\notin\Sigma}H^{1}\be(K_{v^{\le\prime}},G\e)\right].
\]
The {\it capitulation problem for} $\sha^{1}_{\Sigma}(F,G\e)$ (relative to $K\be/\lbe F\e$)
is the problem of describing the kernel of the map $\sha^{1}_{\Sigma}\lbe({\rm Res}_{\lbe G, K\be/\be F}^{(1)})\colon \sha^{1}_{\Sigma}(F,G\e)\to \sha^{1}_{\Sigma}(K,G\e)$
induced by ${\rm Res}_{G,\e K\be/\lbe F}^{(1)}\colon H^{1}\be(F,G\e)\to H^{1}\be(K,G\e)$. By the inflation-restriction exact sequence in Galois cohomology \cite[Proposition 4, p.~100]{aw}, the exact and commutative diagram of abelian groups
\[
\xymatrix{0\ar[r]& \ar[d]^{\sha^{1}_{\Sigma}\lbe({\rm Res}_{\lbe G, K\be/\be F}^{(1)}\lbe)}\sha^{1}_{\Sigma}(F,G\e)\ar[r]& \ar[d]^{{\rm Res}_{G, K\be/\lbe F}^{(1)}}H^{1}\be(F,G\e)\ar[r]& \displaystyle{\prod_{v\notin\Sigma}}H^{1}\be(F_{\be v},G\e)\ar[d]^{\underset{v\notin\Sigma}{\prod}\!{\rm Res}_{G, K_{\be v^{\prime}}\be/\lbe F_{\be v}}^{(1)}}\\
0\ar[r]& \sha^{1}_{\Sigma}\lbe(K,G\e)\ar[r]& H^{1}(K, G\e)\ar[r]& \displaystyle{\prod_{v\notin\Sigma}}H^{1}\be(K_{v^{\le\prime}},G\e)
}
\]
yields the equality
\begin{equation}\label{ksha0}
\krn\e \sha^{1}_{\Sigma}\lbe({\rm Res}_{\lbe G, K\be/\be F}^{(1)}\lbe)=\sha^{1}_{\Sigma}\lbe(\Delta,G(K)),
\end{equation}
where\,\footnote{\e Here we write $G(K)$ for $G_{\be K}\be(K)$ and $G(K_{v^{\le\prime}})$ for $G_{\be K_{\lbe v^{\prime}}}\be(K_{ v^{\prime}}\be)$.} 
\begin{equation}\label{ksha}
\sha^{1}_{\Sigma}\lbe(\Delta,G(K))=
\krn\!\be\left[ H^{1}(\Delta, G(K))\to \displaystyle{\prod_{v\notin\Sigma}}H^{1}\be(\Delta_{\le v^{\prime}},G(K_{v^{\le\prime}}\be))\right].
\end{equation}

Now, by \cite[(3.10)]{ga12}, there exists a canonical exact sequence of abelian groups 
\begin{equation}\label{seq1}
0\to C_{G_{\be F}\lbe,\e F,\e \Sigma}\to H^{1}(S_{\et}, G\e)\to \sha^{1}_{\Sigma}(F,G\e)\to 0.
\end{equation}
A similar exact sequence exists over $K$ if $G\!\times_{\lbe S}\!\spp$ agrees with the identity component of the N\'eron-Raynaud model of $G_{\be K}$ over $\spp$, which is the case here since $\spp$ is \'etale over $S$ (see \cite[\S7.2, Theorem 1(ii), p.~176]{blr} and \cite[${\rm VI_B}$, Proposition 3.3]{sga3}). Thus there exists a canonical exact and commutative diagram of abelian groups
\[
\xymatrix{0\ar[r]& \ar[d]^{j_{G_{\be F},\e K\be/\lbe F,\e\Sigma}}C_{G_{\be F}\lbe,\e F,\e \Sigma}\ar[r]& \ar[d]^{{\rm Res}_{\le G,\e \spp\!\lbe/\be S}^{(1)}}H^{1}(S_{\et}, G\e)\ar[r]& \sha^{1}_{\Sigma}(F,G\e)\ar[d]^{\sha^{1}_{\Sigma}\lbe({\rm Res}_{G,\lle K\be/ \lbe F}^{(1)})}\ar[r]& 0\\
0\ar[r]& C_{G_{\be K}\lbe,\e K,\e \Sigma_{K}}\ar[r]& H^{1}(\spp_{\et}, G\e)\ar[r]& \sha^{1}_{\Sigma}(K,G\e)\ar[r]& 0
}
\]
By \eqref{ksha0}, the preceding diagram induces an exact sequence of abelian groups
\begin{equation}\label{seq2}
0\to \krn\e j_{G_{\be F},\e K\be/\lbe F,\e\Sigma}\to \krn\e {\rm Res}_{\le G,\e \spp\!\lbe/\be S}^{(1)}\to \sha^{1}_{\Sigma}\lbe(\Delta,G(K)).
\end{equation}

In Case (1) above, i.e., $G_{\be F}$ is an invertible $F$-torus, the group $\sha^{1}_{\Sigma}\lbe(\Delta,G(K))$ (which is isomorphic to a subgroup of $H^{1}(F,G\e)$ via the inflation map) is trivial since $H^{1}(F,G\e)$ is trivial by Hilbert's Theorem 90 \cite[Lemma 4.8(a)]{ga10}. In this case \eqref{seq2} reduces to a canonical isomorphism of abelian groups
\[
\krn\e j_{G_{\be F},\e K\be/\lbe F,\e\Sigma}=\krn\e {\rm Res}_{\le G,\e \spp\!\lbe/\be S}^{(1)}.
\]
In Case (2), i.e., $G$ is an abelian scheme over $S$, the groups $C_{G_{\be F}\lbe,\e F,\e \Sigma}$ and $C_{G_{\be K}\lbe,\e K,\e \Sigma_{K}}$ are trivial (as noted previously) and \eqref{seq1} over $F$ and over $K$ yield isomorphisms $H^{1}(S_{\et}, G\e)=\sha^{1}_{\Sigma}(F,G\e)$ and $H^{1}(\spp_{\et}, G\e)=\sha^{1}_{\Sigma}(K,G\e)$. In this case \eqref{seq2} reduces to a canonical isomorphism of abelian groups
\[
\sha^{1}_{\Sigma}\lbe(\Delta,G(K))=\krn\e {\rm Res}_{\le G,\e \spp\!\lbe/\be S}^{(1)}.
\]

Theorem \ref{last} (or, more precisely, sequence \eqref{cool}) and Remark \ref{nik} yield the following statement.

\begin{theorem}\label{fir1} Let $K\be/\be F$ be a quadratic Galois extension of global fields with Galois group $\Delta$ and let $\si$ be a nonempty finite set of primes of $F$ containing the archimedean primes and the non-archimedean primes that ramify in $K$. Set $S=\spec \mathcal O_{\lbe F,\e \Sigma}$ and $\spp =\mathcal O_{\lbe K,\e \Sigma_{K}}$.
\begin{enumerate}
\item[(i)] If $T$ is an invertible torus over $F$ with multiplicative reduction at every point $s\in S$ such that ${\rm char}\, k(s)=2$, then there exists a canonical exact sequence of $2$-torsion abelian groups
\[
\phantom{xxxxx}\frac{\mathcal T^{\e 0}\lbe(S\e)\be\cap\be \mathcal T^{\e 0}(\spp\le)^{2}}{\mathcal T^{\e 0}(S\e)^{2}}\hookrightarrow H^{\le 1}(\Delta, \mathcal T^{\e 0}(\spp)_{\lbe 2})\to\krn\e j_{\e T,\e K\be/\lbe F,\e\Sigma}\twoheadrightarrow\frac{\Psi_{\be N}(2,\mathcal T^{\e 0})}{\mathcal T^{\e 0}\lbe(S\e)\e \mathcal T^{\e 0}\lbe(\spp\le)^{2}}\,,
\]
where $\mathcal T^{\e 0}$ is the identity component of the N\'eron-Raynaud model of \,$T$ over $S$, $j_{\e T,\e K\be/\lbe F,\e\Sigma}\colon C_{T\lbe,\e F,\e \Sigma}\to C_{T\lbe,\e K,\e \Sigma_{K}}$ is the canonical capitulation map for N\'eron-Raynaud $\Sigma$-class groups of $T$ relative to the extension $K\be/\lbe F$ and $\Psi_{\be N}(2,\mathcal T^{\e 0})=\{x\in \mathcal T^{\e 0}\lbe(\spp\le)\colon N_{S^{\lle\prime}\!\lbe/S}(x)\in \mathcal T^{\e 0}(S\e)^{2}\}$.
		
\item[(ii)] If $\mathcal{A}$ is an abelian scheme over $S$, $A=\mathcal{A}_{\e F}$ and $\sha^{1}_{\Sigma}\lbe(\Delta,A(K))$ is the group \eqref{ksha}, then there exists a canonical exact sequence of $2$-torsion abelian groups
\[
\phantom{xxxxx}\frac{A(F\le)\be\cap\be 2\lle A(K\le)}{2\lle A(F\le)}\hookrightarrow H^{\lle 1}\lbe(\Delta, A(K)_{ 2})\to\sha^{1}_{\Sigma}\lbe(\Delta,A(K))\twoheadrightarrow  \frac{\Psi_{\be N}(2,A)}{(\le A(F\le)+2A(K\le))}\,,
\]
where $\Psi_{\be N}(2,A\le)=\{P\in A(K\le)\colon N_{K\be/\lbe F}(P\le)\in 2A(F\le)\}$.
\end{enumerate}
	
\end{theorem}

Setting $T=\bg_{m,\e F}$ in part (i) of the theorem, we obtain the following statement about the classical capitulation kernel $\krn\e j_{K\be/\lbe F,\e\si}$.

\begin{corollary} \label{lac} Let $K\be/\be F$ be a quadratic Galois extension of global fields and let $\si$ be a nonempty finite set of primes of $F$ containing the archimedean primes and the non-archimedean primes that ramify in $K$. Then there exists a canonical exact sequence of finite $2$-torsion abelian groups
\[
0\to\frac{\ofs^{\e *}\be\cap\be (\oks^{\e *}\le)^{2}}{(\ofs^{\e *}\le)^{2}}\to \{\pm 1\}\to \krn\e j_{K\be/\lbe F,\e\si}\to \frac{\Psi_{\lbe N}(K/F,\Sigma\le)}{\ofs^{\e *}\e (\oks^{\e *})^{2}}\to 0,
\]
where
\[
\Psi_{\be N}(\e K/F,\Sigma\e)=\{\e u\in \oks^{\e *}\colon N_{K/F}(u)\in (\ofs^{\e *})^{2}\e\}.
\]
\end{corollary}

\smallskip

\begin{remark}\label{fin} In those cases where the group $\ofs^{\e *}\cap (\oks^{\e *}\le)^{2}/(\ofs^{\e *}\le)^{2}$ has order $2$, the sequence of the corollary yields an isomorphism
\[
\krn\e j_{K\be/\lbe F,\e\si}\simeq\frac{\{\e u\in \oks^{\e *}\colon N_{K/F}(u)\in (\ofs^{\e *})^{2}\e\}}{\ofs^{\e *}\e (\oks^{\e *})^{2}}
\]
which differs from the standard isomorphism
\[
\krn\e j_{K\be/\lbe F,\e\si}\simeq\frac{\{u\in \oks^{\e *}\colon N_{K/F}(u)=1\}}{(1-\tau)\oks^{\e *}}
\]
noted in the Introduction (above $\tau$ is the nontrivial element of $\Delta$). For a variant of Corollary \ref{lac} which applies to arbitrary finite Galois extensions of global fields $K/F$, see \cite[Corollary 4.2]{ga18}.
\end{remark}


\begin{thebibliography}{}

	
\bibitem[1]{an} Anantharaman, S.: Sch\'emas en groupes, espaces homog\`enes et espaces alg\'ebriques sur une base de dimension 1. Mem. Soc. Math. France, {\bf{33}} (1973), 5--79.
	
\bibitem[2]{am} Atiyah, M. and MacDonald, I.: Introduction to commutative algebra,  Addison-Wesley Publishing Company, Inc., Reading, MA (1969).
	
	
\bibitem[3]{aw}	Atiyah, M. and Wall, C.T.C.: Cohomology of groups, in: Algebraic Number Theory (J.W.S. Cassels and A. Fr\"ohlich, Eds.), 94--115, Academic Press, London, 1967.
	
	
\bibitem[4]{bey} Beyl, R.: The connecting morphism in the kernel-cokernel sequence. Arch. der Math. {\bf{32}} (1979), no. 4, 305--308.
		
\bibitem[5]{gfr} Bertapelle, A. and Gonz\'alez-Avil\'es, C.D.
\emph{ The Greenberg functor revisited.} arXiv:1311.0051v4.

\bibitem[6]{blr} Bosch, S., L\"{u}tkebohmert, W. and Raynaud, M.
\emph{ N\'eron Models.} Springer Verlag, Berlin 1989.

	
\bibitem[7]{cgp} Conrad, B., Gabber, O. and Prasad, G.: Pseudo-reductive groups. Second Ed. New Math. Monograps {\bf{26}}, Cambridge U. Press, 2015.
	

\bibitem[8]{c} Conrad, B.: Finiteness theorems for algebraic groups over function fields. 
Comp. Math. {\bf{148}} (2012), no. 2, 555--639.


\bibitem[9]{sga3}  Demazure, M. and
Grothendieck, A. (Eds.): Sch\'emas en groupes. S\'eminaire de
G\'eom\'etrie Alg\'ebrique du Bois Marie 1962-64 (SGA 3). Augmented and
corrected 2008-2011 re-edition of the original by P.Gille and P.Polo.
Available at \url{http://www.math.jussieu.fr/~polo/SGA3}. Reviewed
at \url{http://www.jmilne.org/math/xnotes/SGA3r.pdf}.


\bibitem[10]{ga10} Gonz\'alez-Avil\'es, C.D.\emph{ On N\'eron-Raynaud class groups of tori and the capitulation problem}. J. reine angew. Math. {\bf{648}} (2010), 149--182.
	
	

\bibitem[11]{ga12} Gonz\'alez-Avil\'es, C.D.\emph{ On N\'eron class groups of abelian varieties}. J. reine angew. Math. {\bf{664}} (2012), 71--91.



\bibitem[12]{ga18} Gonz\'alez-Avil\'es, C.D.\emph{ \v{C}ech cohomology and the Capitulation kernel}. Funct. Approx. Comment. Math. doi: 10.7169/facm/1758.
	
\bibitem[13]{sga1} Grothendieck, A.: Rev\^{e}tements \'{E}tales et groupe fondamental. S\'eminaire de G\'eom\'etrie Alg\'ebrique du Bois Marie 1960-61 (SGA 1). Lect. Notes in Math. 224, Springer-Verlag 1973.

	
\bibitem[14]{ega} Grothendieck, A. and Dieudonn\'e, J.: \'El\'ements de g\'eom\'etrie alg\'ebrique. Publ. Math. IHES {\bf{8}} ($=\text{EGA II}$) (1961), {\bf{11}} $\text{III}_{1}$ (1961), {\bf{20}} ($=\text{EGA IV}_{1}$) (1964), {\bf{24}} ($=\text{EGA IV}_{2}$) (1965), {\bf{32}} ($=\text{EGA IV}_{4}$) (1967).


\bibitem[15]{ega1} Grothendieck, A. and
Dieudonn\'e, J.: \'El\'ements de g\'eom\'etrie alg\'ebrique I. Le langage des
sch\'emas. Grund. der Math. Wiss. {\bf{166}} (1971).

\bibitem[16]{sga4} Grothendieck, A. et al.: Th\'eorie des Topos et Cohomologie \'Etale des Sch\'emas. S\'eminaire de G\'eom\'etrie Alg\'ebrique du Bois Marie 1963-64 (SGA 4). Lect. Notes in Math. 270, 305 and 370, Springer-Verlag 1973.


\bibitem[17]{dix} Grothendieck, A.:
Le groupe de Brauer I-III. In: Dix Expos\'es sur la cohomologie
des sch\'emas, North Holland, Amsterdam, 1968, 46--188.


\bibitem[18]{j} de Jong, J. \emph{A result of Gabber} (unpublished). Available at
\url{www.math.columbia.edu/~dejong/papers/2-gabber.pdf}.
		
\bibitem[19]{kps} Knus, M.--A., Parimala, R. and Srinivas, V. Azumaya algebras with involutions. J. Algebra {\bf{130}} (1990), no. 1, 65--82. 
		

		
\bibitem[20]{mi1} Milne, J.S.:\emph{ \'Etale cohomology.}
Princeton University Press, Princeton, 1980.
		
	
\bibitem[21]{mb} Moret-Bailly, L.: Answer to \url{https://mathoverflow.net/questions/285419.}

\bibitem[22]{oes} Oesterl\'e, J.:\emph{ Nombres de Tamagawa et groupes unipotents en caract\'erisque $p$.} Invent. Math. {\bf{78}} (1984), 13-88.

\bibitem[23]{ps} Parimala, R. and Srinivas, V. Analogues of the Brauer group for algebras with involution. Duke Math. J. {\bf{66}} (1992), no. 2, 207--237.

\bibitem[24]{st} Scholz, A. and Taussky, O. Die Hauptideale der kubischen Klassenk\"orper imagin\"ar-quadratische Zahlk\"orper: ihre rechnerische Bestimmung und ihr Einflu{\ss} auf den Klassenk\"orperturm. J. reine angew. Math. {\bf{171}} (1934), 19--41.

\bibitem[25]{sp} The Stacks Project. \url{http://stacks.math.columbia.edu}
	
\bibitem[26]{t} Tamme, G.: Introduction to \'Etale Cohomology. Translated from the German by Manfred Kolster. Universitext. Springer-Verlag, Berlin, 1994.
	
\end{thebibliography}
\end{document}